\documentclass[times]{article}

\usepackage{amsmath,amssymb,amsthm}
\usepackage{authblk}
\usepackage{mathrsfs}

\def\baselinestretch{1.3}  

\usepackage{color}
\usepackage[all]{xy}
\usepackage{epstopdf,tabmac}
\newcommand{\Tableau}[2][sY]{{\text{\tableau[#1]{#2}}}}

\title{Peterson Isomorphism in $K$-theory and Relativistic Toda Lattice}

\author[1]{Takeshi Ikeda}
\author[2]{Shinsuke Iwao}
\author[3]{Toshiaki Maeno}

\affil[1]{Department of Applied Mathematics, Okayama University of Science, 
1-1 Ridai-cho, Kita-ku, Okayama-shi, 
Okayama 700-0005, Japan}
\affil[2]{Department of Mathematics, Tokai University, 
4-1-1 Kitakane, Hiratsuka, Kanagawa-shi, 
Kanagawa 259-1292, Japan}
\affil[3]{
Department of Mathematics, Meijo University, 1-501 Shiogamaguchi, Tenpaku-ku, Nagoya-shi, 
Aichi 468-0073, Japan
}

\date{}

\newtheorem{theorem}{Theorem}[section]
\newtheorem{lemma}[theorem]{Lemma}
\newtheorem{proposition}[theorem]{Proposition}
\newtheorem{corollary}[theorem]{Corollary}

\newtheorem{example}[theorem]{Example}

\newtheorem{remark}[theorem]{Remark}
\newtheorem{definition}[theorem]{Definition}

\newtheorem{conjecture}[theorem]{Conjecture}

\newcommand\PP{\ensuremath{\mathbb{P}}}
\newcommand\ZZ{\ensuremath{\mathbb{Z}}}
\newcommand\CC{\ensuremath{\mathbb{C}}}

\newcommand\zet[1]{\left\vert {#1} \right\vert}

\newcommand{\uni}{\mathrm{uni}}

\usepackage{mathtools}

\makeatletter
\DeclareRobustCommand\widecheck[1]{{\mathpalette\@widecheck{#1}}}
\def\@widecheck#1#2{%
    \setbox\z@\hbox{\m@th$#1#2$}%
    \setbox\tw@\hbox{\m@th$#1%
       \widehat{%
          \vrule\@width\z@\@height\ht\z@
          \vrule\@height\z@\@width\wd\z@}$}%
    \dp\tw@-\ht\z@
    \@tempdima\ht\z@ \advance\@tempdima2\ht\tw@ \divide\@tempdima\thr@@
    \setbox\tw@\hbox{%
       \raise\@tempdima\hbox{\scalebox{1}[-1]{\lower\@tempdima\box
\tw@}}}%
    {\ooalign{\box\tw@ \cr \box\z@}}}

\def\mWord#1{\text{\ $\mWord@A{#1}$}} %
\def\mWord@A#1{%
\bgroup
\let\\=\crcr
\def\,{& &}
\def\u##1{\underline{##1}}
\def\d##1{\dot{##1}}
  \vcenter{\offinterlineskip
      \halign{
       &\vrule width 0pt height 8pt depth 3pt {\hfil${##}$\hfil} \!\! & \!\!{##},\! \cr #1 \crcr
      }
  }
\egroup
}

\makeatother
\newcommand{\Array}[1]{\left(\!\mWord{#1}\!\right)}
\newcommand\V[1]{\pmb{#1}}
\newcommand\ouni{\mathscr{O}_{\mathrm{uni}}}

\begin{document}
\maketitle
\begin{abstract} 
The $K$-homology ring of the affine Grassmannian of $SL_n(\CC)$  
was studied by Lam, Schilling, and Shimozono. It 
is realized as a certain concrete Hopf subring 
of the ring of symmetric functions.  
On the other hand, for the quantum $K$-theory of the flag variety $Fl_n$, 
Kirillov and Maeno provided 
a conjectural presentation based on the results obtained by Givental and Lee. 
We construct an explicit 
birational morphism between the spectrums of these two rings. Our method relies on  
Ruijsenaars's relativistic Toda lattice with unipotent initial condition.
From this result, we obtain a $K$-theory analogue of the so-called Peterson isomorphism for (co)homology. 
We provide a conjecture on the detailed relationship between the Schubert bases, and, in particular, we determine the image of Lenart--Maeno's quantum Grothendieck polynomial
associated with a Grassmannian permutation. 
\end{abstract}

  \section{Introduction}
Let $Fl_n$ be the variety of complete flags $V_\bullet=(V_1\subset \cdots\subset V_n=\mathbb{C}^n)$ in $\mathbb{C}^n$, which is a homogeneous  
space $G/B$, where $G=SL_n(\mathbb{C})$ and $B$ is the Borel subgroup of the 
upper triangular matrices in $G.$
Let $K(Fl_n)$ be the Grothendieck ring of coherent 
sheaves on $Fl_n.$ 
Givental and Lee \cite{givental2003quantum}
studied the quantum $K$-theory
$QK(Fl_n)$, 
which is a 
ring defined as a deformation of  $K(Fl_n)$ (see \cite{giventalWDVV, Lee2001} for general construction of quantum $K$-theory).
Similar to Givental--Kim's presentation (see \cite{givental1995quantum, KimAnnals1999}) for the
quantum cohomology ring $QH^*(Fl_n)$, Kirillov and Maeno \cite{Kirillov-Maeno} provided a conjectural presentation for $QK(Fl_n)$, which 
we denote temporally by $\mathcal{QK}(Fl_n).$ 
Let $Gr_{SL_n}=G\left(\mathbb{C}(\!(t)\!)\right)/G(\mathbb{C}[[t]])$ be the {\it affine Grassmannian\/} of $G=SL_n$, whose 
$K$-homology $K_*(Gr_{SL_n})$ has a natural structure of  
a Hopf algebra. 
Lam, Schilling, and Shimozono \cite{LSS} constructed a Hopf isomorphism between $K_*(Gr_{SL_n})$ 
and the subring $\Lambda_{(n)}:=\CC[h_1,\dots,h_{n-1}]$ of the ring $\Lambda$ of symmetric functions.
The first main result of this paper is an explicit 
ring isomorphism between 
$K_*(Gr_{SL_n})$ and
$\mathcal{QK}(Fl_n)$ after appropriate localization. 

The corresponding result in (co)homology, for a semisimple linear algebraic group $G$, is called the {\it Peterson
isomorphism\/}; this result was presented in lectures by Peterson in MIT, 1997,
and published in a paper by Lam and Shimozono \cite{Lam2010} on torus-equivariant and parabolic settings. 

{We also provide a conjecture that describes a detailed 
correspondence between the Schubert bases for
$K_*(Gr_{SL_n})$ and
$\mathcal{QK}(\mathcal{F}l_n)$ (Conjecture \ref{tildeg}). 
}

Our method of constructing the isomorphism
relies on Ruijsenaars's relativistic Toda lattice  \cite{ruijsenaars1990relativistic}. 
{Note that a} similar approach to the original Peterson isomorphism for $SL_n$ by solving the non-relativistic Toda lattice was given by 
Lam and Shimozono \cite{lam2011double, lamshimo2010toda}
and by Kostant \cite{Kostant1996Toda}. 

It is natural to ask how $K$-theoretic Peterson isomorphisms for a general semisimple linear algebraic group $G$ should be constructed. 
There are some results indicating that an approach to this problem using integrable systems would be fruitful. The relativistic Toda lattice associated with any root system 
was introduced by Kruglinskaya and Marshakov in \cite{KruglinskayaMarshakov2015}. 
The commuting family of $q$-difference Toda operators given by Etingof, as well as 
general such operators constructed by using a quantized enveloping 
algebra $U_q(\mathfrak{g})$ of a complex semisimple Lie algebra $\mathfrak{g}$, 
were discussed by Givental and Lee \cite{givental2003quantum}. 
As for $K$-homology of the affine Grassmannian: 
Bezrukavnikov, Finkelberg, and Mirkovi\'c \cite{bezrukavnikov2005equivariant} showed that the spectrum of the $G(\mathbb{C}[[t]])$-equivariant $K$-homology ring of ${Gr}_G$ is naturally identified with the {\it universal centralizer\/} of the Langlands dual group of $G$. 

\subsection{Relativistic Toda lattice}
The {\it relativistic Toda lattice\/}, introduced by Ruijsenaars \cite{ruijsenaars1990relativistic},
is a completely integrable Hamiltonian system. 
This system can be viewed as an isospectral deformation of the 
Lax matrix $L=AB^{-1}$ with 
\begin{equation}
A=\begin{pmatrix}
z_1 & -1 &  & & \\
 & z_2 & -1&  & \\
 & & \ddots & \ddots & \\
 &  &    &z_{n-1}& -1\\
 & & & & z_n
\end{pmatrix},\quad 
B=
\begin{pmatrix}
1 &  &  && \\
-{Q_1z_1}& \!\!\!\!\!\!1& &  & \\
 & \ddots & \ddots & & \\
 &  &    &1\;\;\;\;& \\
 & &&\!\!\!\!\!\!\!\!\!\!\!\!\!\!-Q_{n-1}z_{n-1}  &\!\!\!\!1
\end{pmatrix},\label{eq:L}
\end{equation}
where $z_i$ and $Q_j$ are complex numbers.
The relativistic Toda lattice is a partial differential equation with independent 
variables $t_1,\ldots,t_{n-1}$
 expressed in the Lax form:
\begin{equation}
\frac{d L}{dt_i}=[L,(L^i)_{<}]\quad (i=1,\ldots,n-1),
\end{equation}
where $(L^i)_{<}$ denotes the strictly lower triangular part of 
$L^i.$
Consider the characteristic polynomial of $L$:
\[
\Psi_L(\zeta)=
\det(\zeta\cdot 1-L)=\zeta^n+\sum_{i=1}^n (-1)^i
F_i(z,Q)\zeta^{n-i}.\] 
More explicitly, 
we have
\begin{equation}
F_i(z_1,\ldots,z_n,Q_1,\ldots,Q_{n-1})=\sum_{\substack{I\subset \{1,\ldots,n\}\\\;\# I=i}}\prod_{j\in I}z_j\prod_{j\in I,\;j+1\notin I }(1-Q_j),
\end{equation}
where $Q_n:=0$.
Let $\gamma=(\gamma_1,\ldots,\gamma_{n})\in \mathbb{C}^{n}$. In this paper, we assume
$\gamma_n=1$ and consider 
the {\it isospectral variety\/}
\begin{equation}
Z_\gamma:=
\{(z,Q)\in \mathbb{C}^{2n-1}\;|\;F_i(z,Q)=\gamma_i\;(1\leq i\leq n)\}.\label{eq:defZgamma}
\end{equation}
This is an affine algebraic variety with coordinate ring
\begin{equation*}
\mathbb{C}[Z_\gamma]=\mathbb{C}[z_1,\ldots,z_n,Q_1,\ldots,Q_{n-1}]
/I_\gamma,\quad I_\gamma=\langle F_i(z,Q)-\gamma_i\;(1\leq i\leq n)\rangle.
\end{equation*}
Note that $Z_\gamma$ is isomorphic to a locally closed subvariety of $SL_n(\mathbb{C})$
(see \S \ref{ssec:Lax}).

\subsection{Quantum $K$-theory of flag variety}
A particularly interesting case occurs when 
\[\gamma_i={\binom{n}{i}}\quad (1\leq i\leq n).\]
Then, the corresponding Lax matrix $L$ has
characteristic polynomial $(\zeta-1)^n$.
In fact, in this case, $L$ is principal unipotent in the sense that  
it has only one Jordan block such that the eigenvalues are all $1$, so we say the corresponding isospectral variety is {\it unipotent\/} and denote it by 
 $Z_\uni.$
 Let $\mathcal{L}_i$ be the {\it tautological line bundle\/} whose fiber 
 over $V_\bullet\in Fl_n$ is $V_i/V_{i-1}.$ 
\begin{conjecture}[Kirillov--Maeno \cite{Kirillov-Maeno}]
\label{conj:KM}
There is a canonical ring isomorphism 
\begin{equation}
QK(Fl_n)\simeq \mathbb{C}[Z_{\mathrm{uni}}],\label{eq:LM}
\end{equation}
with $z_i$ identified with the class of 
the tautological 
line bundle $\mathcal{L}_i$.
\end{conjecture}


Givental and Lee \cite{givental2003quantum} studied 
a certain generating function of the $\mathbb{C}^* \times SL_n(\mathbb{C})$-equivariant Euler characteristic of a natural family of line bundles 
on the {\it quasimap spaces\/}  
from $\mathbb{P}^1$ to $Fl_n$. The function is a formal power series in $Q=(Q_1,\ldots,Q_{n-1})$ and its copy $Q'$, depending on $q$, the coordinate
of $\mathbb{C}^*$, and $(\Lambda_1,\ldots,\Lambda_n)$ with $\Lambda_1\cdots \Lambda_n=1$, the coordinates
of the maximal torus of $SL_n(\mathbb{C})$, respectively.
They proved that the generating function is the eigenfunction
of the finite $q$-difference Toda operator. {A relation to the Toda
lattice has been studied further by Braverman and Finkelberg \cite{BF1,BF2}.}
Accordingly, it has been expected that there is a 
finite $q$-difference counterpart of quantum $D$-module (\cite{GiventalHomGeom1995, GiventalEqGW1996}, see also \cite{KimAnnals1999}) giving the structure 
of $QK(Fl_n)$;
however, to the best of the authors' knowledge, the connection between the multiplication in $QK(Fl_n)$ and the $q$-difference system is still uncertain. 
We hope that recent work by Iritani, Milanov, and Tonita \cite{IritaniMilanovTonita} gives 
an explanation of this connection, and that ultimately Conjecture \ref{conj:KM} will be proved.  

We add some remarks on {a certain} finiteness property of the quantum $K$-theory. Note that the ring is originally defined as a $\mathbb{C}[[Q]]$-algebra (see \cite{giventalWDVV, Lee2001}). For 
cominuscule $G/P$, such as the Grassmannian $Gr_d(\mathbb{C}^n)$ of $d$-dimensional subspaces of $\mathbb{C}^n$, 
it was shown by Buch, Chaput, Mihalcea, and Perrin \cite{BCMP2013,BCMP2016} that multiplicative structure constants
for the (quantum) Schubert basis $[\mathcal{O}_{X_w}]$ are polynomial
in $Q_i$ (see also Buch and Mihalcea \cite{BuchMihalcea2011} for earlier results on the Grassmannian). 
Thus, for such a variety, the $\mathbb{C}[Q]$-span of 
the Schubert classes forms a subring.
  So far, it is not known whether or not the finiteness holds for $Fl_n$ in general. In the above 
 conjecture, $QK(Fl_n)$ should be 
interpreted as the $\mathbb{C}[Q]$-span of the Schubert classes.

With this conjecture in mind, in this paper, we will denote $\mathbb{C}[Z_{\mathrm{uni}}]$
by $\mathcal{QK}(Fl_n)$ in distinction to ${QK}(Fl_n)$.

\begin{remark}(added in proof) After this article was submitted, a preprint \cite{ACT}
by Anderson--Chen--Tseng appeared, in which they proved the finiteness of the torus equivariant $K$-ring of $Fl_n$
and the relations $F_i(z,Q)=e_i(\Lambda_1,\ldots,\Lambda_n)$. This implies that Conjecture \ref{conj:KM} is true.
\end{remark}

\subsection{Dual stable Grothendieck polynomials}
Let 
$\Lambda$ denote the complexified\footnote{It is possible to work over the integers; however, we use complex coefficients throughout the paper because we use many coordinate rings of complex algebraic varieties.}
 ring of symmetric functions (see \cite{macdonald1998symmetric}). If we denote by $h_i$ the $i$th {\it complete symmetric function\/},
$\Lambda$ is the polynomial ring $\mathbb{C}[h_1,h_2,\ldots].$ It has the so-called {\it Hall inner product\/} $\langle\cdot,\cdot\rangle: \Lambda\times\Lambda\rightarrow \mathbb{C}$ and a standard Hopf algebra structure.

For each partition $\lambda=(\lambda_1\geq \lambda_2\geq \cdots\geq \lambda_\ell)$ of length 
$\ell$, the {\it Schur functions\/}
$s_\lambda\in \Lambda$ are defined by 
\begin{gather}
s_\lambda=\det(h_{\lambda_i+j-i})_{1\leq i,j\leq \ell}.
\end{gather}
The {\it stable Grothendieck polynomial\/} $G_\lambda$ 
is given as the sum of {\it set-valued tableaux\/} of shape $\lambda$,
and they
form a basis of a completed ring $\hat{\Lambda}$ of symmetric functions (see Buch \cite{BuchLR} for details).
The 
{\it dual stable Grothendieck polynomials\/} $\{g_\lambda\}$ due to 
Lam and Pylyavskyy \cite{lamply2007} are defined 
by $\langle g_\lambda,G_\mu\rangle=\delta_{\lambda,\mu}.$
It was shown in \cite{lamply2007} that
$\{g_\lambda\}$ is identified with the $K$-homology Schubert basis of the
infinite Grassmannian (see \cite[\S 9.5]{lamply2007} for a more precise statement). 
Shimozono and Zabrocki \cite{ShimozonoZabrocki} proved a determinant formula for 
$g_\lambda$.
The following formula for 
$g_\lambda$
is also available:
\begin{equation}
g_\lambda=\det
\left(
\textstyle 
\sum_{m=0}^\infty(-1)^m{1-i\choose m}h_{\lambda_i+j-i-m}
\right)_{1\leq i,j\leq \ell}=s_\lambda+\mbox{lower degree terms},\label{detg}
\end{equation}
where $h_k=0$ for $k<0$, and
${1-i\choose m}=(1-i)\cdots(1-i-m+1)/m!.$
\begin{remark}
It is straightforward to see the equivalence of formula (\ref{detg}) and the one in \cite{ShimozonoZabrocki} (see also \cite[Equation (4)]{LascouxNaruse2014}), so we omit the details. 
\end{remark}
\subsection{Lam--Schilling--Shimozono's presentation for $K_*(Gr_{SL_n})$}
In \cite{LSS}, Lam, Schilling and Shimozono
showed that the $K$-homology of the affine Grassmannian $K_*(Gr_{SL_n})$ \cite[\S 5.1]{LSS}
can be realized as a subring in the affine $K$-theoretic nil-Hecke algebra of Kostant--Kumar \cite{KostantKumar1990K-theory}. 
Let us denote by 
$\Lambda_{(n)}:=\mathbb{C}[h_1,\ldots,h_{n-1}]$ the subring of $\Lambda$ generated by $h_1,\ldots,h_{n-1}.$
In \cite{LSS}, the following ring isomorphism was established:
\begin{equation}
K_*(Gr_{SL_n})\simeq \Lambda_{(n)}.\label{eq:LSSisom}
\end{equation}
Note that 
$K_*(Gr_{SL_n})$ is equipped with a Hopf algebra structure 
coming from the based loop space $\Omega SU(n),$ and 
the above is a Hopf isomorphism
with the Hopf algebra structure on $\Lambda_{(n)}$ induced from the canonical one   
on $\Lambda.$

\subsection{$K$-theoretic Peterson isomorphism}
Let $Z_\uni^\circ$ be the Zariski open set of $Z_\uni$ 
defined as the complement of the divisor defined by $Q_1\cdots Q_{n-1}= 0.$
Thus the coordinate ring $\CC[Z_\uni^\circ]$ is 
a localization of $\CC[Z_\uni]=\CC[z,Q]/I_\uni$ by $Q_i\;(1\leq i\leq n-1)$.
In view of Conjecture \ref{conj:KM}, we define
\begin{equation}
\mathcal{QK}(Fl_n)_{\mathrm{loc}}:=\CC[Z_\uni^\circ].
\end{equation}

For the affine Grassmannian side, we define 
\begin{equation}
K_*(Gr_{SL_n})_{\mathrm{loc}}
:=\Lambda_{(n)}[\sigma_i^{-1},\tau_i^{-1}\;(1\leq i\leq n-1)],
\end{equation}
where 
\begin{equation}
\tau_i=g_{R_i},\quad \sigma_i=\sum_{\mu\subset R_i}g_{\mu}
\quad(1\leq i\leq n-1),\label{eq:deftau}
\end{equation}
where $R_i$ is the rectangle\footnote{Our notation of $R_i$ is conjugate (transpose) to 
the one used in \cite{lamshimo2010toda}.} with $i$ rows and $(n-i)$ columns. 
We set $\tau_0=\sigma_0=\tau_n=\sigma_n=1.$

\begin{example} Let $n=3.$ Then we have
\begin{equation*}
\tau_1=h_2,\quad
\tau_2=h_1^2-h_2+h_1,\quad
\sigma_1=h_2+h_1+1,\quad
\sigma_2=h_1^2-h_2+2h_1+1.
\end{equation*}
\end{example}
The main result of this paper is the following theorem. 

\begin{theorem}[Corollary \ref{cor:KPisom}, Propositions \ref{prop:zQST}, \ref{lemma:kappa}, \ref{prop:kappasigma}]\label{thm:main-intro}
There is an isomorphism of rings
\begin{equation}\label{eq:K-perterson}
\Phi_n:
\mathcal{QK}(Fl_n)_{\mathrm{loc}}\overset{\sim}{\longrightarrow}
K_*(Gr_{SL_n})_{\mathrm{loc}},
\end{equation}
given by 
\begin{equation}
z_i\mapsto \frac{\tau_i\sigma_{i-1}}{\sigma_i\tau_{i-1}}\quad(1\leq i\leq n),
\quad
Q_i\mapsto \frac{\tau_{i-1}\tau_{i+1}}{\tau_i^2}\quad(1\leq i\leq n-1).\label{eq:zQ}
\end{equation}
\end{theorem}

\begin{remark}
Let $p_i$ be the $i$th power sum symmetric function in $\Lambda$. With the 
identification $t_i={p_i}/{i}\;(1\leq i\leq n-1)$, the formulas (\ref{eq:deftau}), (\ref{eq:zQ}) 
give an explicit solution of the relativistic Toda lattice that is a rational function in time variables $t_1,\ldots,t_{n-1}.$
Thus, if we substitute the rational expression of $(z,Q)$ in $t_j$ into $F_i(z,Q)$, then such a complicated rational function is equal to $\binom{n}{i}.$ This remarkable identity can be understood as a consequence of the isospectral property of the relativistic Toda lattice.
\end{remark}

\subsection{Method of construction}
The ring homomorphism $\Phi_n$ is obtained by solving
the relativistic Toda lattice. Our method of 
solving the system is analogous to the one employed by Kostant \cite{Kostant1979Toda} for the finite non-periodic Toda lattice.
Let $C_\gamma$ be the companion matrix of the (common) characteristic
polynomial of $L\in Z_\gamma.$
We denote the centralizer of $C_\gamma$ by $\mathfrak{X}_\gamma$, 
which is the affine space of dimension $n.$
We can construct a birational morphism $\alpha$ from $Z_\gamma$ to 
$\mathbb{P}(\mathfrak{X}_\gamma).$
On the open set of $\mathbb{P}(\mathfrak{X}_\gamma)$ defined as the complement of
the divisors given by $\tau_i$, $\sigma_i$, we can construct
the inverse morphism $\beta$ to $\alpha.$ When $\gamma$ is unipotent, the associated 
isomorphism between the coordinate rings of these open sets is nothing but $\Phi_n.$
%


\subsection{Quantum Grothendieck polynomials}
The {\it quantum Grothendieck polynomials\/} $\mathfrak{G}_w^Q$ \cite{lenart2006quantum} 
of
Lenart and Maeno
are a family of polynomials in the variables
$x_1,\ldots,x_n$ and quantum parameters $Q_1,\ldots,Q_{n-1}$
indexed by permutations $w\in S_n$ (see \S \ref{ssec:GroP} for the definition).
Note that $x_i$ is identified with $1-z_i$ (the first Chern class of the dual line bundle of $\mathcal{L}_i$).  
It was conjectured in \cite{lenart2006quantum} that the polynomial $\mathfrak{G}_w^Q$ represents the quantum
Schubert class $[\mathcal{O}_{X_w}]$ in $\mathbb{C}[z,Q]/I_\uni$ 
under the conjectured isomorphism\footnote{The main result of \cite{lenart2006quantum} is a
Monk-type formula for $\mathfrak{G}_w^Q$.
It is worthwhile to emphasize that the formula 
is proved logically independent from Conjecture \ref{conj:KM}. In fact, the formula holds in the polynomial ring of $x_i$ and $Q_i$.}
 (\ref{eq:LM}).
 
 \subsection{$K$-theoretic $k$-Schur functions}
{
In the sequel we use the notation $k = n-1$.
Let $\mathcal{B}_k$ denote the set of $k$-{\it bounded partitions} $\lambda$, i.e. the partitions such that $\lambda_1\leq k.$  
The $K$-{\it theoretic $k$-Schur functions\/}
$\{g_\lambda^{(k)}\}\subset  \Lambda_{(n)}\simeq K_*(Gr_{SL_n})$ are indexed by $\lambda\in \mathcal{B}_k$ and defined in \cite{LSS} as the dual basis of the Schubert basis of the Grothendieck ring $K^*(Gr_{SL_n})$
of the thick version of the affine Grassmannian (see \cite{Kashiwarathick} for Kashiwara's construction of the thick flag variety of a Kac--Moody Lie algebra). 
{
The highest degree component of $g_\lambda^{(k)}$ is the $k$-{\it Schur function\/}  $s_\lambda^{(k)}$
introduced by Lapointe, Lascoux, and Morse \cite{LLMkSchur}  (see also \cite{kSchurBook} and references therein).}

 We note that $g_\lambda^{(k)}$ is equal to $g_\lambda$ if $k$ is sufficiently large \cite[Remark 7.3]{LSS}. }
Morse \cite{Morse2012}\footnote{The conjecture after Property 47 in \cite{Morse2012}.}  conjectured that for a partition $\lambda$ such that $\lambda\subset R_i$ for some 
$1\leq i\leq k$, we have 
\begin{equation}
g_{\lambda}^{(k)}=g_\lambda,\label{recrec}
\end{equation} thus 
in particular, 
$
g_{R_i}^{(k)}=g_{R_i}=\tau_i. 
$ 
This conjecture is relevant for our considerations. 
Note that the counterpart of this conjecture for $k$-Schur functions was proved in \cite{LapointeMorse2007}; that is,
$s_{\lambda}^{(k)}=s_\lambda$ if $\lambda\subset R_i$ for some $1\leq i\leq k.$
In particular, we have $s_{R_i}^{(k)}=s_{R_i}$.
It is worth remarking $s_{R_i}$ arise as the ``$\tau$-functions'' of the finite 
non-periodic Toda lattice with nilpotent initial condition.

 \subsection{Image of the quantum Grothendieck polynomials}
We are interested in the image of $\mathfrak{G}_w^Q$ by $\Phi_n.$

Let 
$\mathrm{Des}(w)= \{i\;|\;1\leq i\leq n-1,\, \;w(i)>w(i+1)\}$ denote the set of descents of $w\in S_n$.
A permutation $w\in S_n$ is $d$-{\it Grassmannian\/} if $\mathrm{Des}(w)=\{d\}$; such elements are
in bijection with partition $\lambda=(\lambda_1,\ldots,\lambda_d)$ such that $\lambda\subset R_d$
. If $\lambda\subset R_d$, then the corresponding 
$d$-Grassmannian permutation $w_{\lambda,d}$ (see Remark \ref{rem:GrassPerm}) has length 
$|\lambda|=\sum_{i=1}^d\lambda_i.$
Let $\lambda^\vee:=(n-d-\lambda_d,\ldots,n-d-\lambda_1).$ 
\begin{theorem} [cf. Corollary \ref{cor:Gr}]\label{thm:Grass} Let $\lambda$ be a permutation such that $\lambda\subset R_d.$ 
Then we have
\begin{equation}
\Phi_n(\mathfrak{G}_{w_{\lambda,d}}^Q)=\frac{g_{\lambda^\vee}}{\tau_d}.\label{PhiG}
\end{equation}
\end{theorem}

For a general permutation $w$, we will provide a conjecture on the image of $\mathfrak{G}_w^Q$.
Let $\lambda: S_n\rightarrow \mathcal{B}_k$ be 
a map defined by Lam and Shimozono \cite[\S 6]{lamshimo2010toda} (see \S \ref{ssec:lambda} below).
In order to state the conjecture we also need an involution $\omega_k$ on $\mathcal{B}_k$,
$\mu\mapsto \mu^{\omega_k}$, 
called the $k$-{\it conjugate\/} (\cite{LapointeMorse2005JCTA}).
The image of the map $\lambda$ consists of elements in $\mathcal{B}_k$ that are
$k$-{\it irreducible}, that is, those $k$-bounded partition $\mu=(1^{m_1}2^{m_2}\cdots (n-1)^{m_{n-1}})$ such that $m_i\leq k-i\;(1\leq i\leq n-1)$.
Let $\mathcal{B}_k^*$ denote the set of all $k$-irreducible $k$-bounded partitions. 
Note that $\mathcal{B}_k^*$ is preserved by $k$-conjugate. 
Let us denote by $S_n^*$ the subset $\{w\in S_n\;|\;w(1)=1\}$ of $S_n$.
We know that $\lambda$ gives a bijection from 
$S_n^*$ to  $\mathcal{B}_k^*$ (see \S \ref{ssec:lambda} below).

\begin{conjecture}\label{tildeg}
Let $w$ be in $S_n$.
There is a polynomial $\tilde g_{w}
\in \Lambda_{(n)}$ such that 
\begin{equation}
\Phi_n(\mathfrak{G}_w^Q)=\frac{ \tilde g_{w}}{\prod_{i\in \mathrm{Des}(w)}\tau_i}.
\end{equation}
Forthermore, $\tilde g_{w}$ satisfies the following properties:

(i) If $\lambda(w)=\lambda(w')$ for $w,w'\in S_n$, then we have
\begin{equation*}
\tilde g_w=\tilde g_{w'}.
\end{equation*}

(ii) For $w\in S_n$, we have 
\begin{equation}
\tilde g_w=g^{(k)}_{\lambda(w)^{\omega_k}}+\sum_{\mu}
a_{w,\mu}g_{\mu}^{(k)},\quad a_{w,\mu}\in \ZZ,
\end{equation}
where $\mu$ runs for all elements in $\mathcal{B}_k^*$ such that $|\mu|<|\lambda(w)|.$

(iii) $(-1)^{|\mu|-|\lambda(w)|}a_{w,\mu}$ is a non-negative integer.
\end{conjecture}

The counterpart of Conjecture \ref{tildeg} in the (co)homology case was established in \cite{lamshimo2010toda}, where the quantum Schubert polynomial $\mathfrak{S}_w^q$ of Fomin, Gelfand, and Postnikov \cite{FominGelfandPostnikov1997}
is sent by the original Peterson isomorphism to the fraction, whose numerator is the single $k$-{Schur function} associated with 
$\lambda(w)^{\omega_k}$, and 
the denominator is the products of $s_{R_i}^{(k)}$ such that $i\in \mathrm{Des}(w).$
Note in the formula in \cite{lamshimo2010toda}, the numerator is the $k$-Schur function associated with ${\lambda(w)}$ without $k$-conjugate by reason of the convention.


If $w\in S_n$ is $d$-Grassmannian for some $d$, and $w=w_{\mu,d}$ with $\mu
\subset R_d$, 
then from Theorem \ref{thm:Grass} 
we have
\begin{equation*}
\tilde g_w=g_{\mu^\vee}.
\end{equation*}
Since we know that $\mu^\vee=\lambda(w_{\mu,d})^{\omega_k}$ (Lemma \ref{lem:kconj} below), if (\ref{recrec}) is true, 
Conjecture \ref{tildeg} holds for all Grassmannian permutations $w$.
In the early stage of this work, we expected that $\tilde g_w$ is always a single $K$-$k$-Schur function, however, this is not the case; for example
we have 
\begin{equation*}
\tilde g_{1423}=g_{2,1,1}^{(3)}-g_{2,1}^{(3)}.
\end{equation*}

\subsection{Further discussions}
Let us assume that 
Conjecture \ref{tildeg} is true and discuss its possible implications. The property 
(iii) says $\{\tilde g_w\}_{w\in S_n^*}$ and $\{g_{\nu}^{(k)}\}_{\mu\in \mathcal{B}_k^*}$ are two bases of the same space, and the transition matrix between the bases  
is lower unitriangular. 
This suggests that $\{\tilde g_w\}_{w\in S_n^*}$ is a part of an important basis
of $\Lambda_{(n)}$ different from the $K$-$k$-Schur basis.
One possibility of such basis will be the following: 
Let us
denote the function $\tilde g_w$ by 
$\tilde g_{\nu}^{(k)}$ with $\nu=\lambda(w)^{\omega_k}\in\mathcal{B}_k^*$.
For a general $k$-bounded partition $\mu$, we can uniquely write it as
$
\mu=\nu\cup \bigcup_{i=1}^{n-1}R_i^{e_i},\;
\nu\in \mathcal{B}_k^*,\; e_i\geq 0.
$
Then we define
\begin{equation*}
\tilde g_{\mu}^{(k)}=\tilde g_{\nu}^{(k)}\cdot \tau_1^{e_1}\cdots
\tau_{n-1}^{e_{n-1}}.
\end{equation*}
One sees that $\tilde g_{\mu}^{(k)}\;( \mu\in \mathcal{B}_k)$ form a
basis of $\Lambda_{(n)}.$


%


%
%

In the proof of the isomorphism in \cite{Lam2010}, they work in $T$-equivariant ($T$ is the maximal torus of $G$) settings, and first give the module isomorphism $\psi$ from $H_*^T(Gr_G)_{\mathrm{loc}}$ to
$QH^*_T(G/B)_{\mathrm{loc}}$, and  
next prove the $\psi$-preimage of the quantum Chevalley formula. Since the quantum Chevalley formula uniquely characterizes $QH^*_T(G/B)$ due to
 a result of Mihalcea \cite{MihalceaDuke2007}, we know that $\psi$ is a ring isomorphism. 
 In our situation, we proved the ring isomorphism $\Phi_n$ (Theorem \ref{cor:KPisom}) without using 
the quantum Monk formula of $\mathcal{QK}(Fl_n)$ (cf. Lenart-Postnikov \cite{LP}). 
Thus the basis $\{\tilde g_\nu^{(k)}\}_{\nu\in \mathcal{B}_k}$ should satisfy the corresponding formula in $\Lambda_{(n)}.$ 
These issues will be studied further elsewhere.



\bigskip
\subsection{Organization.} In Sections 2--4 of this paper, we give the $K$-theoretic Peterson morphism and prove it is an isomorphism (Theorem \ref{cor:KPisom}). In Sections 5--6,
we calculate the image of quantum Grothendieck polynomials
associated with Grassmannian permutations (Theorem \ref{thm:Grass}).
In Section 7, we discuss some details of Conjecture \ref{tildeg}.

In Section \ref{sec:KP}, we state the main results of the first main part of the paper. 
We construct a birational 
morphism $\alpha$ from $Z_\gamma$ to $\mathbb{P}(\mathscr{O}_\gamma)$ with 
$\mathscr{O}_\gamma:=\mathbb{C}[\zeta]/(\zeta^n+\sum_{i=1}^{n-1}
(-1)^i\gamma_i\zeta^{n-i}).$
We also describe the open sets
of both $Z_\gamma$ and $\mathbb{P}(\mathscr{O}_\gamma)$
that are isomorphic as affine algebraic varieties. 
The complement of the open part of $\mathbb{P}(\mathscr{O}_\gamma)$
is a divisor given by explicitly defined functions $T_i,S_i$.
The main construction of Section 2 is the definition of the map $\Phi_n$ in Theorem \ref{cor:KPisom}. 
In fact, the isomorphism statement of Theorem \ref{cor:KPisom}
is given as Corollary \ref{cor:KPisom}, which is the unipotent case of 
Theorem \ref{thm:isom}. In Section 3, we 
 prove Theorem \ref{thm:isom}. We also give a more conceptual description of the map
$\alpha$ and its inverse $\beta$. 
Section \ref{ssec:Phi} includes the formula of $\Phi_n$ in terms of the functions $T_i,S_i$.
In Section 4 we determine the precise form of the $\tau$-functions $T_i, S_i$ in terms of dual stable Grothendieck polynomials, thus completing the proof of Theorem \ref{cor:KPisom}.

In Section 5, we summarize the second main result. The aim is to calculate the quantum Grothendieck polynomials
associated with Grassmannian permutations.
We give a version  $\widehat Q_d$ of the quantization map from the $K$-ring 
of the Grassmannian $Gr_d(\mathbb{C}^n)$ 
to $\mathcal{QK}(Fl_n)$ by using our $K$-Peterson isomorphism.
The main statement of Section 5 is the compatibility of $\widehat Q_d$ and Lenart--Maeno's 
quantization map $\widehat Q$ 
with respect to the embedding 
$K(Gr_d(\mathbb{C}^n))\hookrightarrow K(Fl_n)$ 
(Theorem \ref{thm:quantize}). As a corollary to this, we obtain Theorem \ref{thm:Grass}. 
Section 6 is devoted to the proof of Theorem \ref{thm:quantize}.
In Section 7, we explain some details of Conjecture 2 and give some examples of calculations.

\setcounter{equation}{0}
\section{Construction of $K$-theoretic Peterson Isomorphism}\label{sec:KP}
In this section, we state our main construction.

Let $\gamma_i\;(1\leq i\leq n-1)$ be any complex numbers and set $\gamma_n=1.$
Let 
\begin{equation*}
f_\gamma(\zeta)=\zeta^n+\sum_{i=1}^{n}(-1)^i\gamma_i \zeta^{n-i}.
\end{equation*}
Let 
$
\mathscr{O}_\gamma$ denote the quotient ring 
$\mathbb{C}[\zeta]/(f_\gamma(\zeta)).
$
We also consider $\mathscr{O}_\gamma$ as 
an affine space. 
We use the following notation of minor determinants for an $n\times n$ matrix
$X=(x_{ij})_{1\leq i,j\leq n}$:
\begin{equation}
\xi_{i_1,\ldots,i_r}^{j_1,\ldots,j_r}(X)=\det(x_{i_a,j_b})_{1\leq a,b\leq r}.
\end{equation}
Let $\Delta_{i,j} = \Delta_{i,j}(z,Q)
= \xi^{1,2,\dotsc,\hat{j},\dotsc,n}_{1,2,\dotsc,\hat{i},\dotsc,n}(\zeta B(z,Q) - A(z)).$

\begin{definition}\label{def:alpha} We define the map 
$\alpha: Z_\gamma\rightarrow \mathbb{P}(\mathscr{O}_\gamma)$
sending 
$L\in Z_\gamma$
to 
$[\Delta_{1,1}]$.
\end{definition}

\begin{example}If $n=3$ then $\Delta_{1,1}$ is given as
\begin{equation*}
\Delta_{1,1}=\zeta^2+(Q_2z_2-z_2-z_3)\zeta+z_2z_3.
\end{equation*}
\end{example}
This is a handy definition of $\alpha$.  
A more conceptual description of the map $\alpha$ is 
given in the next section, where $\mathscr{O}_\gamma$ is 
interpreted as the centralizer of the companion matrix
$C_\gamma$ of the characteristic polynomial $f_\gamma.$
In fact, we will construct the inverse $\beta$ of $\alpha$ defined
on an open set of $\mathbb{P}(\mathscr{O}_\gamma)$,
which is the counterpart of the map Kostant \cite{Kostant1979Toda} defined for the ordinary finite Toda lattice. 

The case of our interest, as was noted above, is 
$\gamma_i=\binom{n}{i}$
that is equivalent to $\Psi_L(\zeta)=(\zeta-1)^n.$ We call this parameter 
{\it unipotent}
and denote the corresponding isospectral variety by  $Z_\uni$.
Recall that we denote 
\begin{equation}
\mathcal{QK}(Fl_n)=\mathbb{C}[Z_\uni],
\end{equation}
which is our working definition of the quantum $K$-theory of $Fl_n.$

We will show below that $\alpha$ is a birational morphism of algebraic varieties.
We also describe open parts that are isomorphic via the map $\alpha$ explicitly. 
As the corresponding isomorphism between the coordinate rings, we obtain the 
$K$-theory analogue of the Peterson isomorphism.

\begin{definition}\label{def:tau}
Fix a linear isomorphism 
$\pmb{c}: \mathscr{O}_\gamma\rightarrow \mathbb{C}^n.$ For $0\leq j\leq n$, and $\varphi\in \mathscr{O}_\gamma$, let $\pmb{a}_j=\pmb{c}(\zeta^j),\;\pmb{b}_j=\pmb{c}(\varphi\zeta^j).$ Define for $1\leq i\leq n$,
\begin{align*}
T_i(\varphi)&=|\pmb{b}_0,\pmb{b}_1,\cdots,\pmb{b}_{i-1},\pmb{a}_{i-1},\cdots,\pmb{a}_{n-2}|,\\
S_i(\varphi)&=|\pmb{b}_0,\pmb{b}_1,\cdots,\pmb{b}_{i-1},\pmb{a}_{i},\cdots,\pmb{a}_{n-1}|.
\end{align*}
\end{definition}
Note that a different choice of $\pmb{c}$ yields a change 
$T_i\mapsto c^iT_i,\; S_i\mapsto c^i S_i$ with a nonzero constant $c\in \mathbb{C}^*.$ Such change does not effect the following constructions, however, 
we choose $\pmb{c}$ so that $|\pmb{a}_0,\ldots,\pmb{a}_{n-1}|=1.$
For each $i$, both $T_i$ and $S_i$ are homogenous polynomial functions in $\mathbb{C}[\mathscr{O}_\gamma]$ of degree $i.$
Let $Y_\gamma=\mathbb{P}(\mathscr{O}_\gamma)$
and  
define a Zariski open set
\begin{equation}
Y_\gamma^\circ=\{[\varphi]\in\mathbb{P}(\mathscr{O}_\gamma)
\;|\;
T_i(\varphi)\neq 0,\;S_i(\varphi)\neq 0\;(1\leq i\leq n) 
\}.
\end{equation}
Our first main result is the following.

\begin{theorem}\label{thm:isom} The map
$\alpha$ gives an isomorphism from $Z_\gamma^\circ$ to $Y_\gamma^\circ$
as affine algebraic varieties.
\end{theorem}

Now we apply this to unipotent case, namely the case when
$\gamma_i=\binom{n}{i}.$ We choose  
$\pmb{c}: \mathscr{O}_\gamma\rightarrow \mathbb{C}^n,\;
\varphi\mapsto {}^t(c_0,c_1,\ldots,c_{n-1})$
as
\begin{equation}
\varphi=
\sum_{i=0}^{n-1}(-1)^ic_i\cdot(\zeta-1)^i. \label{eq:ci}
\end{equation}
Then we have $T_n=S_n=c_0^n$. So $Y_\uni^\circ$ is
an open subvariety of the affine 
open set $U_0$ of $\mathbb{P}(\mathscr{O}_\uni)$ defined by $c_0\neq 0$. 
We identify the coordinate ring $\mathbb{C}[c_1/c_0,\ldots,c_{n-1}/c_0]$ of 
$U_0$  with $\Lambda_{(n)}=\mathbb{C}[h_1,\ldots,h_{n-1}]$ by 
\begin{equation}
h_i=c_i/c_0\quad (1\leq i\leq n-1).
\end{equation}
Using this identification, we will prove (see \S \ref{sec:tau})
\begin{equation}
\tau_i=T_i/c_0^i,\quad \sigma_i=S_i/c_0^i\quad(1\leq i\leq n-1).
\label{eq:taudet}
\end{equation}
 Via the isomorphism $K({Gr}_{SL_n})\simeq \Lambda_{(n)}$, we have 
$\mathbb{C}[Y_\uni^\circ]=K({Gr}_{SL_n})[\tau_i^{-1},\sigma_i^{-1}].$
\begin{corollary}\label{cor:KPisom} We have the following isomorphism of rings:
\begin{equation*}
\Phi_n:
\mathcal{QK}(Fl_n)[Q_i^{-1}(1\leq i\leq n-1)]
\overset{\sim}{\longrightarrow} K({Gr}_{SL_n})[\tau_i^{-1},\sigma_i^{-1}(1\leq i\leq n-1)].
\end{equation*}
\end{corollary}

The explicit formula of $\Phi_n$ (the second statement of Theorem \ref{cor:KPisom}) will be derived below in Proposition \ref{prop:zQST} in \S \ref{ssec:Phi}
together with Propositions \ref{lemma:kappa} and \ref{prop:kappasigma}.
\setcounter{equation}{0}
\section{Proof of Theorem \ref{thm:isom}}
This section is devoted to the proof of Theorem \ref{thm:isom}.

\subsection{Gauss decomposition}

Let $\pmb{B}$ (resp. $\pmb{B}_{\!-}$) denote the Borel subgroup of $GL_n(\mathbb{C})$
consisting of upper (resp. lower) triangular matrices. Let $\pmb{N}_{\!-}$ (resp. $\pmb{N}$) denote the subgroup consisting of the  
unipotent lower (resp. upper) triangular matrices.  

\begin{proposition}\label{prop:AB}
A square matrix $X$ of size $n$ can be  expressed 
as $X=X_{+}\cdot X_{-}$ with $X_{+}\in \pmb{B}$, and $X_{-}\in \pmb{N}_{\!-}$, 
if and only if
$\xi_{i+1,\ldots,n}^{i+1,\ldots,n}(X)\neq 0$ for $0\leq i\leq n-1.$
\end{proposition}
\begin{proof}
This is the factorization known as the Gauss or the LU-decomposition.  
The result is standard. See \cite{Noumi9th} for example. 
\end{proof}

Let $\sigma$ denote the matrix $\sum_{i=1}^{n-1}E_{i+1,i}+E_{1,n}$
which represents the cyclic permutation $(1,\ldots,n).$
Let $\varepsilon:=\mathrm{diag}(1,-1,1,\ldots,(-1)^{n-1}).$

\begin{proposition}\label{prop:RU}
Let $X$ be a square matrix $n$ such that $x_{1,n}\neq 0$. Then $X$ can be  expressed 
as $X=U^{-1} R$ with $R=(r_{ij})\in \pmb{B}\sigma,\;U=(u_{ij})\in \pmb{N}_{\!-}\varepsilon$, 
if and only if
$\xi_{1,\ldots,i-1,i}^{1,\ldots,i-1,n}(X)\neq 0$ for $2\leq i\leq n-1.$
Moreover, if such decomposition exists, we have
\begin{equation}
r_{i+1,i}=(-1)^{i+1}\frac{\xi^{1,\ldots,i,n}_{1,\ldots,i,i+1}(X)}{\xi_{1,\ldots,i-1,i}^{1,\ldots,i-1,n}(X)}\quad (1\leq i\leq n-1).\label{eq:r-ratio}
\end{equation}
\end{proposition}

\subsection{The variety $Z$ of Lax matrices}\label{ssec:Lax}
Let $J=\sum_{i=1}^{n-1}E_{i,i+1}.$
Let $Z$ denote the set of matrices $L$ in $SL_n(\mathbb{C})$ satisfying the following conditions:
\begin{itemize}
\item[$(\mathrm{Z_1})$]: $L+J$ is a lower triangular matrix,
\item[$(\mathrm{Z_2})$]:  all entries of $L^{-1}$ further down the second subdiagonal are zero,
\item[$(\mathrm{Z_3})$]:  $\xi_{i+1,\ldots,n}^{i+1,\ldots,n}(L)\neq 0$ for $1\leq i\leq n-1.$
\end{itemize}

Let $T$ be the subgroup of $(\mathbb{C}^{*})^n$ 
consisting of $(z_1,\ldots,z_n)$ such that $z_1\cdots z_n=1.$ 
\begin{proposition}\label{prop:LAB}
The map $
T\times \mathbb{C}^{n-1}
\rightarrow Z$ defined by sending 
$(z,Q)$ with $z\in T,\;Q=(Q_1,\ldots,Q_{n-1})\in\mathbb{C}^{n-1}$ 
to $L=AB^{-1}$ with 
\begin{equation}
A=\sum_{i=1}^nz_iE_{i,i}-J,\quad
B=1-\sum_{i=1}^{n-1}Q_i z_iE_{i+1,i},\label{eq:defL}
\end{equation}
is an isomorphism of algebraic varieties.
\end{proposition}
\begin{proof} It is easy to see $L=AB^{-1}$ given by (\ref{eq:defL}) satisfies
($\mathrm{Z}_1$). $L^{-1}$ is given as follows:
\begin{equation}
\begin 
{pmatrix} \frac{1}{z_1}&{\frac {1}{z_{{1}}z_{{2}}}}&\cdots&\cdots&{\frac {1}{z_{{1}}z_{{2}}\cdots z_{{n}}}}\medskip
\\-Q_{{1}}&-{\frac {Q_{{1}}-1}{z_{{2}}}}&
\cdots&\cdots
&-\frac {Q_{{1}}-1}{z_2\cdots z_n}
\\ 0&-Q_{{2}}&\ddots&\cdots&\vdots\\ 
0&0&\ddots&-{\frac {Q_{{n-2}}-1}{z_{{n-1}}}}&-{\frac {Q_{{n-2}}-1}{z_{{n-1}}z_n}}
\\ 0&0&0&-Q_{{n-1}}&
-{\frac {\small{Q}_{{n-1}}-1}{z_{{n}}}}\end {pmatrix},\label{eq:Linv}
\end{equation}
and thereby ($\mathrm{Z}_2$) holds.
To see $L$ satisfies ($\mathrm{Z}_3$) we only need to notice that $L$ is factorized as in Proposition \ref{prop:AB}.
We construct the inverse map by using
Proposition \ref{prop:AB}.
Let $L\in Z$. We decompose it as $L=AB^{-1}$, where $A\in \pmb{B},\;
B^{-1}\in \pmb{N}_{\!-}.$  
Let $M=L^{-1}.$ 
Define $Q_i=-M_{i+1,i}$ and $z_i=M_{1,i}/M_{1,i-1}$ with $M_{1,0}=1.$
It is straightforward, by using $(\mathrm{Z_1})$
and $(\mathrm{Z_2})$, to check $A,B$ are given by (\ref{eq:defL}).
\end{proof}

Thus $Z$ is the affine variety whose coordinate ring $\mathbb{C}[Z]$ is 
$\mathbb{C}[z,Q]/(z_1\cdots z_n-1)$. 
We define the subset $Z^\circ$ of $Z$ by imposing the condition: 
\begin{itemize}
\item[$(\mathrm{Z_4})$]: $Q_i\neq 0\;(1\leq i\leq n-1).$
\end{itemize}
Note that $Z_\gamma$ defined by (\ref{eq:defZgamma}) is a closed 
subvariety of $Z$.
Let $Z_\gamma^\circ=Z_\gamma\cap Z^\circ.$

\subsection{Centralizer of $C_\gamma$}
Let $C_\gamma$ denote the companion matrix of $f_\gamma(\zeta).$
Explicitly $C_\gamma=J+\sum_{i=1}^n (-1)^{i-1}\gamma_i E_{n,n-i+1}.$
Let $\mathfrak{X}_\gamma$ denote
the set of all matrices that commute with $C_\gamma$.
Any $X\in \mathfrak{X}_\gamma$ is uniquely
expressed as a polynomial
\begin{equation}
X=
\sum_{i=0}^{n-1}\alpha_i\cdot C_\gamma^i\quad (\alpha_i\in \mathbb{C})\label{eq:phiC}
\end{equation}
in $C_\gamma$ of degree at most $n-1.$
This fact can be checked directly. 
In view of the Cayley-Hamilton theorem, the map from $\mathbb{C}[\zeta]$
sending $\varphi(\zeta)$ to
$\varphi(C_\gamma)$ induces an isomorphism 
$\mathscr{O}_\gamma\rightarrow \mathfrak{X}_\gamma$
of affine varieties. 
In the following, we identify $\mathscr{O}_\gamma$
with $\mathfrak{X}_\gamma$ via this map.
Let $Y_\gamma^\circ$ denote the subset of $\mathbb{P}(\mathscr{O}_\gamma)$ such that the representatives $\varphi\in \mathscr{O}_\gamma-\{0\}$ satisfy the following conditions :
\begin{itemize}
\item[$(\mathrm{Y}_0)$]: $\varphi(C_\gamma)$ is invertible.
\item[$(\mathrm{Y}_1)$]:  $(1,n)$ component of
$\varphi(C_\gamma)$ is non-zero.
\item[$(\mathrm{Y}_2)$]: $T_i(\varphi)\neq 0\;(1\leq i\leq n-1).$ 
\item[$(\mathrm{Y}_3)$]: $S_i(\varphi)\neq 0\;(1\leq i\leq n-1).$ 

\end{itemize}
Note that
$(\mathrm{Y}_1)$ is equivalent to the condition that
$\varphi(\zeta)$ can be chosen so that it has degree $n-1$.
We say 
$\varphi$ is {\it normalized\/} if it is monic of degree $n-1$. 

\begin{remark} It should be natural to consider the set of elements of $\mathbb{P}(\mathcal{O}_\gamma)$ satisfying 
($\mathrm{Y}_1$) as the centralizer of $[C_\gamma]$
in $PGL_n(\mathbb{C}),$ the Langlands dual group of $SL_n(\mathbb{C}).$
\end{remark}

$T_i(\varphi)$ and $S_i(\varphi)$ (Definition \ref{def:tau}) are given in terms of the matrix $\varphi(C_\gamma)$ as follows.
\begin{lemma}\label{lem:tau-det} We have the following:

(1) $T_i(\varphi)=(-1)^{n-i}\xi^{1,\ldots,i-1,n}_{1,\ldots,i-1,i}(\varphi(C_\gamma)).$

(2) $S_i(\varphi)=\xi^{1,\ldots,i}_{1,\ldots,i}(\varphi(C_\gamma)).$

\end{lemma}
\begin{proof}
Let us define $\pmb{c}:\mathscr{O}_\gamma\rightarrow
\mathbb{C}^n$ by sending polynomial to the 
reminder with respect to $f_\gamma$ and  expand it with basis $1,\zeta,\ldots,\zeta^{n-1}$, in particular 
$\zeta^i\mapsto \pmb{e}_{i+1}\;(1\leq i\leq n-1).$
Then the $i$th row of $\varphi(C_\gamma)$ is ${}^t\pmb{b}_{i-1},$
and $\pmb{a}_i=\pmb{e}_{i+1}$. Now the formulas are easily obtained. 
\end{proof}

\subsection{Construction of $\alpha:Z_\gamma^\circ\rightarrow Y_\gamma^\circ$}

\begin{proposition}\label{prop:RUexist} Let $(z,Q)\in Z_\gamma$ and denote $L(z,Q)$ by $L$.

(1) There is a matrix $R=(r_{ij})_{1\leq i,j\leq n}$ in $\pmb{B}\sigma$, unique up to scalar, such that 
$LR=RC_\gamma$. 
Moreover, for such a matrix $R$, by multiplying some non-zero constant if needed, we have
\begin{align}
\det(R)&=(-1)^{n(n-1)/2}Q_1^{n-1}Q_2^{n-2}\cdots Q_{n-1},\label{eq:detR}\\
{r_{i+1,i}}&=(-1)^{i-1} Q_1\cdots Q_i\quad (1\leq i\leq n-1).\label{eq:QQQ}
\end{align}
(2) There is a unique matrix $U$ in $\pmb{N}_{\!-}\varepsilon$ such that 
$LU=UC_\gamma$. 
\end{proposition}
\begin{proof}
(1) Let $\Delta_{i,j}$ denote the
$(i,j)$-minor of $\zeta B- A,$
i.e. $\Delta_{i,j}=\xi_{1,\ldots,\hat{i},\ldots,n}
^{1,\ldots, \hat{j},\ldots,n}(\zeta B-A).$
 It is straightforward to show the following: 
\begin{align}
(-1)^{1+j}\Delta_{1,j}&=
\bigl(\zeta^{n-1}
+\cdots+(-1)^jz_{j+1}\cdots z_n\cdot\zeta^{j-1}\bigr)\cdot\prod_{i=1}^{j-1}Q_iz_i,\label{eq:w1}\\
\Delta_{n,j}&=\zeta^{j-1}+\cdots +(-1)^{j-1}z_1\cdots z_{j-1}.\label{eq:w2}
\end{align}
Note in particular that $\Delta_{1,1}$ is monic of degree $n-1,$
and $\Delta_{n,1}=1.$

If we define a vector 
\begin{equation}
\pmb{v}_{-}:={}^t(\Delta_{1,1},-\Delta_{1,2},\ldots,(-1)^{n-1}\Delta_{1,n})\
\end{equation} in $\mathbb{C}[\zeta]^n$, 
then by the Laplace expansion theorem
we have 
\begin{equation}
(\zeta B- A)\pmb{v}_{-}={}^t(\det(\zeta B- A),0,\ldots,0)
={}^t(f_\gamma(\zeta),0,\ldots,0).\label{eq:etaB-A}
\end{equation}

Let $\pmb{w}_{-}:=B\,\pmb{v}_{-}$. We apply the natural projection $\mathbb{C}[\zeta]\rightarrow
\mathscr{O}_\gamma$ to both hand sides of (\ref{eq:etaB-A}). Then we have the following equation in $\mathscr{O}_\gamma^n$:
\begin{equation}
(\zeta\cdot 1-L)\,\pmb{w}_{-}=\pmb{0}.\label{eq:eta-L}
\end{equation}
We can write $\pmb{w}_{-}=R\,\pmb{v}_0$ by a unique matrix $R\in M_n(\mathbb{C}),$
with $\pmb{v}_0={}^t(1,\zeta,\ldots,\zeta^{n-1})$.
Noting that $\zeta \cdot \pmb{v}_0=C_\gamma \pmb{v}_0,$
 (\ref{eq:eta-L}) is $(RC_\gamma-LR)\pmb{v}_0=\pmb{0},$
and thereby we have $RC_\gamma-LR=0.$

We need to check that $R\in \pmb{B}\sigma.$ 
If we write $\pmb{v}_{-}=R_0\pmb{v}_0$ with $R_0\in GL_n(\mathbb{C})$ then we can see from (\ref{eq:w1}) that $R_0$ is upper triangular such that
$n$th column of $R_0$ is ${}^t(1,Q_1z_1,Q_1Q_2z_1z_2,\ldots,Q_1\cdots Q_{n-1}z_1\cdots z_{n-1})$.
Consequently, $R=BR_0$ has the desired form.
We also have (\ref{eq:QQQ}) because $(j,j)$ entry of $R_0$ is
$(-1)^jz_{j+1}\cdots z_n\prod_{i=1}^{j-1}Q_iz_i$.

If $R'$ also satisfies $R'C_\gamma-LR'=0$, then by writing
$R=X\sigma, R'=X'\sigma$ with $X,X'\in \pmb{B}$, 
$X^{-1}X'$ commute with $\sigma C_\gamma \sigma^{-1}.$
By direct calculations, one sees that this commutativity means $X^{-1}X'$ is a scalar matrix. 

Next we show (\ref{eq:detR}). Since $B$ is unitriangular, we have $\det(R)=\det(R_0).$
This is calculated by (\ref{eq:w1}) as 
\begin{equation*}
\prod_{j=1}^n\bigl((-1)^jz_{j+1}\cdots z_n\cdot\prod_{i=1}^{j-1}Q_iz_i\bigr)
=(-1)^{n(n-1)/2}(z_1\cdots z_n)^{n-1}
Q_1^{n-1}Q_2^{n-2}\cdots Q_{n-1}.
\end{equation*}
Since $z_1\cdots z_n=1$, we have (\ref{eq:detR}).

(2) We define $\pmb{v}_{+}\in \mathbb{C}[\zeta]^n$ by
\begin{equation}
\pmb{v}_{+}:={}^t(\Delta_{n,1},-\Delta_{n,2},\ldots,(-1)^{n-1}\Delta_{n,n}).\
\end{equation}
If we set $\pmb{w}_{+}=B\,\pmb{v}_{+} $ then we have
$(\zeta\cdot 1-L)\,\pmb{w}_{+}=0$ in $\mathscr{O}_\gamma^n$.
We write $\pmb{w}_{+}=U\pmb{v}_0.$ Then we have $LU=UC_\gamma$ by the same reason.
From (\ref{eq:w2}) we see that 
$U\in \pmb{N}_{\!-}\varepsilon.$ 
The uniqueness is straightforward. 
\end{proof}

Now we construct a morphism $\alpha^\circ:Z_\gamma^\circ\rightarrow Y_\gamma^\circ.$
Let $(z,Q)\in Z_\gamma^\circ$ and $L=L(z,Q)$ be the corresponding 
Lax matrix (Proposition \ref{prop:LAB}). Let $U,R$ be matrices constructed in Proposition \ref{prop:RUexist}.
Note that $R$ is invertible because we assume $(\mathrm{Z}_4)$ and we have 
(\ref{eq:detR}).
We have $L=RC_\gamma R^{-1}=UC_\gamma U^{-1}.$ 
So $U^{-1}R\in\mathfrak{X}_\gamma\simeq \mathscr{O}_\gamma.$ 
Let us write $U^{-1}R=\varphi(C_\gamma)$ by 
a polynomial $\varphi$ in $\zeta$ as (\ref{eq:phiC}). 
Set $\alpha^\circ(L):=[\varphi]\in \mathbb{P}(\mathscr{O}_\gamma)$. From the form of matrices $R,U$, one sees that
the $(1,n)$ entry of $U^{-1}R$ is $1$, so $\varphi$ is monic of degree $n-1$. Thus
$(\mathrm{Y}_1)$ holds for $\varphi(C_\gamma).$
We have $(\mathrm{Y}_0)$ since $R$ is invertible. 
Since $\varphi(C_\gamma)$ is factorized as $U^{-1}R$ 
as in Proposition \ref{prop:RU}, 
$(\mathrm{Y}_2)$ holds in view of 
Lemma \ref{lem:tau-det} (1).
Finally we check $(\mathrm{Y}_3)$.  We use Lemma \ref{lem:tau-det} (2).
From similar calculation of the proof of (\ref{eq:detR}), 
the $i$-th principal minor of 
$U^{-1}R\;(=\varphi(C_\gamma))$ is written as
\begin{equation*}
(-1)^iz_{i+1}\cdots z_nQ_1^{i-1}Q_2^{i-2}\cdots Q_{i-1},
\end{equation*}
which is non-zero by assumption $(\mathrm{Z}_3).$
Thus $\alpha^\circ(L)$ is an element of $Y_\gamma^\circ.$

\begin{proposition} We have
$\alpha|_{Z_\gamma^\circ}=\alpha^\circ.$
\end{proposition}
\begin{proof} Since
$\varphi(C_\gamma)=U^{-1}R$, we have in $\mathscr{O}_\gamma^n$ the following:
\begin{equation*}
\pmb{w}_{-}=R\,\pmb{v}_0=U\varphi(C_\gamma)\pmb{v}_0=
U\varphi(\zeta)\pmb{v}_0
=\varphi(\zeta)U\pmb{v}_0
=\varphi(\zeta)\pmb{w}_{+},
\end{equation*}
and thereby $\pmb{v}_{-}=\varphi(\zeta)\pmb{v}_{+}.$
By comparing the first component of the both sides of this equality, we have
$\Delta_{1,1}=\varphi(\zeta).$ 
\end{proof}
\subsection{Construction of $\beta: Y_\gamma^\circ\rightarrow Z_\gamma^\circ$}

Let $\varphi$ be a normalized element in $Y_\gamma^\circ$.
Since we assume $(\mathrm{Y}_2)$, from Proposition \ref{prop:RU}, we have unique $R,U$ such that
\begin{equation}
\varphi(C_\gamma)=U^{-1}R.\label{eq:phi-decomp}
\end{equation}
From $(\mathrm{Y}_0)$, we see that $R$ is invertible. 
Now we have $UC_\gamma U^{-1}=RC_\gamma R^{-1}=:L$. Set
 $\beta([\varphi]):=L.$
We claim that $L$ is an element of $Z^\circ.$
It is easy to check that $L$ satisfies $(\mathrm{Z}_1)$ and  $(\mathrm{Z}_2)$ by using $L=UC_\gamma U^{-1}$ and $L=RC_\gamma R^{-1}$
respectively.
$(\mathrm{Z}_3)$ follows from $(\mathrm{Y}_3)$ by the following.

\begin{lemma}\label{lem:Y3Z3} Let $\varphi$ and $L$ as above. Then we have
\begin{equation}
\xi_{i+1,\ldots,n}^{i+1,\ldots,n}(L)=-\frac{S_i(\varphi)}{T_i(\varphi)}\;(1\leq i\leq n-1).
\label{eq:detLprinc}
\end{equation}
\end{lemma}
\begin{proof}
Note that for any element $R$ in $\pmb{B}\sigma$, we have
 \begin{equation*}
 \xi_{i+1,\ldots,n}^{i+1,\ldots,n}(RC_\gamma R^{-1})
=(-1)^{n-1}\frac{\xi_{1,\ldots,i}^{1,\ldots,i}(R)}{r_{21}r_{32}\cdots r_{i,i-1} }.
\end{equation*}
Now using (\ref{eq:r-ratio}) we have 
$
r_{21}r_{32}\cdots r_{i,i-1}=(-1)^{n-i+i(i+1)/2}\, T_i(\varphi),
$
and 
$$\xi_{1,\ldots,i}^{1,\ldots,i}(R)=(-1)^{i(i-1)/2}\xi_{1,\ldots,i}^{1,\ldots,i}(U^{-1}R).$$
Thus we have  the result from Lemma \ref{lem:tau-det} (2).
\end{proof}

Finally we need to show $(\mathrm{Z}_4)$. It is assured by the fact that $R$ is invertible. Let $L=L(z,Q).$
Note that $\det(R)=(-1)^{n-1}r_{2,1}r_{3,2}\cdots r_{n,n-1}\neq 0.$
Note also that $(i+1,i)$ entry of $L^{-1}$ is $-Q_i$ (cf. (\ref{eq:Linv})), which is the ratio
$r_{i+1,i}/r_{i,i-1}$ with $r_{1,0}=-1$, and thus non-zero.
Thus $L$ is in $Z^\circ.$
Since $L$ is conjugate to $C_\gamma$ we have $\beta([\varphi])=L\in Z_\gamma^\circ.$

\begin{proposition} The morphisms 
$\alpha^\circ$ and $\beta$ are inverse to each other.
\end{proposition}
\begin{proof}
Let $L\in Z_\gamma^\circ$.
We take $U,R$ such that 
$L=UC_\gamma U^{-1}=RC_\gamma R^{-1}$.
Then $\varphi(\zeta)=\Delta_{1,1}$ satisfies
$\varphi(C_\gamma)=U^{-1}R.$
Then $\alpha(L)=[\varphi]$ and $\beta([\varphi])=UC_\gamma U^{-1}=L.$
On the other hand, let $[\varphi]\in Y_\gamma^\circ.$
We assume $\varphi$ is normalized. We have $\varphi(C_\gamma)=U^{-1}R$. 
Then $\beta([\varphi])=UC_\gamma U^{-1}=RC_\gamma R^{-1}=:L.$
Then $\alpha^\circ(L)$ is given by $U^{-1}R.$
\end{proof}

\subsection{Explicit form of $\Phi_n$}\label{ssec:Phi}

\begin{proposition}\label{prop:zQST}
Let $\Phi_n:\mathbb{C}[Z_\gamma^\circ]\rightarrow \mathbb{C}[Y_\gamma^\circ]$
be the associated isomorphism of coordinate rings. Then we have
\begin{align}
\Phi_n(z_i)&=\frac{T_i(\varphi)S_{i-1}(\varphi)}{S_i(\varphi)T_{i-1}(\varphi)},\label{eq: Phiz}\\
\Phi_n(Q_i)&=\frac{T_{i-1}(\varphi)T_{i+1}(\varphi)}{T_i^2(\varphi)}.\label{eq: PhiQ}
\end{align}
\end{proposition}
\begin{proof}
It is easy to show
$\xi_{i,\ldots,n}^{i,\ldots,n}(L)=z_i\cdots z_n.$ Then 
(\ref{eq: Phiz}) follows from (\ref{eq:detLprinc}) and 
Lemma \ref{lem:tau-det}.
To prove (\ref{eq: PhiQ}) we use (\ref{eq:r-ratio}), (\ref{eq:QQQ}) to have 
$$
Q_i=-\frac{r_{i+1,i}}{r_{i,i-1}}
=
\frac{\xi_{1,\ldots,i-1,i}^{1,\ldots,i-1,n}(X)}{\xi_{1,\ldots,i-2,i-1}^{1,\ldots,i-2,n}(X)}
\cdot
\frac{\xi_{1,\ldots,i-3,i-2}^{1,\ldots,i-3,n}(X)}{\xi_{1,\ldots,i-2,i-1}^{1,\ldots,i-2,n}(X)}.
$$
with $X=\varphi(C_\gamma).$
Then from Lemma \ref{lem:tau-det} (1), we have (\ref{eq: PhiQ}).
\end{proof}

\begin{remark}
We observe that each component of $L$ expressed as a rational
function of $T_i(\varphi)$'s and $S_i(\varphi)$'s has no pole along the divisor defined by $S_i(\varphi)=0.$
This is mysterious because $z_i$ has pole along the divisor. We do not know any explanation of this phenomena. 
\end{remark}
  \section{Dual stable Grothendieck Polynomials as $\tau$-functions}\label{sec:tau}
In this section we prove (\ref{eq:taudet}). It shows that the functions $T_d,S_d$ 
defined as the determinants
are, in the unipotent case, written in terms of the dual Grothendieck polynomials associated to the rectangle $R_d.$
\subsection{The determinant formula for dual stable Grothendieck polynomials}

Let $U_0\subset \PP(\ouni)$ be the affine subspace defined by $c_0\neq 0$.
We identify the coordinate ring $\CC[U_0]$ with the ring $\Lambda_{(n)}=\CC[h_1,\dots,h_{n-1}]$ through the identification $c_i/c_0\leftrightarrow h_i$ (see (\ref{eq:ci}) for the definition of $c_i$).

We fix $d$ such that $1\leq d\leq n$. 
For $f_1,f_2,\dots,f_d\in \ouni$, define $[f_1,f_2,\dots,f_d]\in \CC[U_0]=\Lambda_{(n)}$ by the formula
\begin{align*}
&[f_1,f_2,\dots,f_d](\varphi)\\
&=c_0^{-d}\cdot
\zet{
\V{c}(f_1\varphi),
\V{c}(f_2\varphi),
\dots,
\V{c}(f_d\varphi),
\V{c}(1),
\V{c}(\zeta),
\dots,
\V{c}(\zeta^{n-d-1})
}
\qquad
(\varphi\in U_0).
\end{align*}
If $f_j\varphi\in U_0$ is expressed as
$f_j\varphi=\sum_{i=0}^{n-1}a_{i,j}(\zeta-1)^i\mod{(\zeta-1)^n}$,
we have 
\begin{equation}
[f_1,\dots,f_d](\varphi)=(-1)^{d(n-d)}c_0^{-d}\cdot\det(a_{n-d+i-1,j})_{i,j=1}^d.\label{eq:detC}
\end{equation}
Note that for $m\geq 0$, $\zeta^{-m}$ makes sense as an element of $\ouni$
since $(1-\zeta)^n=0$.
We have $
\zeta^{-m}=\sum_{l=0}^{n-1} (-1)^l\textstyle{\binom{m+l-1}{l}}(\zeta-1)^l.
$

\begin{lemma}\label{lemma:g}
For a partition $\lambda\subset R_d$, we have 
\begin{equation}
  [
  (1-\zeta)^{n-1-\lambda_d}\zeta^{-d+1},
  (1-\zeta)^{n-2-\lambda_{d-1}}\zeta^{-d+2},
  \dots,
  (1-\zeta)^{n-d-\lambda_1}
  ]
=g_\lambda.\label{eq:g}
\end{equation}
\end{lemma}
\begin{proof} 
For $a,b\geq 0$, 
we write $(1-\zeta)^a\zeta^{-b}\varphi=\sum_{i=0}^{n-1}c_i^{a,b}(\zeta-1)^i \mod (\zeta-1)^n$. Then we have
$$
c_{i}^{a,b}=\sum_{m=0}^{\infty}(-1)^{i+m}{{-b\choose m}}h_{i-a-m},$$ 
where
we understand $h_m=0$ for $m<0.$
By using (\ref{eq:detC}), the left hand side of (\ref{eq:g}) can be calculated as follows:
\begin{align*}
&(-1)^{d(n-d)}\det\left(
c_{n-d+j-1}^{n-i-\lambda_{d-i+1},d-i}
\right)_{i,j=1}^{d}\\
&=(-1)^{d(n-d)}\det
\left(
\textstyle 
\sum_{m=0}^\infty(-1)^{n-d+j-1+m}{i-d\choose m}h_{\lambda_{d-i+1}+j-(d-i+1)-m}
\right)_{i,j=1}^d\\
&
=
(-1)^{d(n-d)+\sum_{j=1}^d(n-d+j-1)+d(d-1)/2}\det
\left(
\textstyle 
\sum_{m=0}^\infty(-1)^{m}{1-i\choose m}h_{\lambda_i+j-i-m}
\right)_{i,j=1}^d \\
&=g_\lambda.
\end{align*}
\end{proof}

\begin{proposition}\label{lemma:kappa} We have
$\tau_d=g_{R_d}.$
\end{proposition}
\begin{proof} 
Note that $M_\zeta\pmb{c}(f)=\pmb{c}(f\zeta)$ with $M_\zeta:=E_n+J$. 
We have 
\begin{align*}
\tau_d
&=c_0^{-d}\cdot\left|
\pmb{c}(\varphi),\dots,\pmb{c}(\varphi\zeta^{d-1}),\pmb{c}(\zeta^{d-1}),\dots,\pmb{c}(\zeta^{n-2})
\right|\\
&=
c_0^{-d}\cdot\left|M_\zeta\right|^{d-1}
\left|
\pmb{c}(\varphi\zeta^{-(d-1)}),\dots,\pmb{c}(\varphi),\pmb{c}(1),\dots,\pmb{c}(\zeta^{n-d-1})
\right|\\
&=[\zeta^{-(d-1)},\zeta^{-(d-2)},\dots,1]\\
&=[(1-\zeta)^{d-1}\zeta^{-(d-1)},(1-\zeta)^{d-2}\zeta^{-(d-2)},\dots,1].
\end{align*}
The last equality follows from the fact that 
$
(1-\zeta)^{m}\zeta^{-m}-\zeta^{-m}\;(m\geq 1)
$ is a linear combination of $1,\zeta^{-1},\ldots,\zeta^{-m+1}$.
Then from Lemma \ref{lemma:appendix} we obtain
$\tau_d=g_{R_d}$.
\end{proof}

\subsection{A useful family of determinants}

Here we introduce some useful family of determinant formulas for symmetric polynomials which is suitable for our representation of $K$-theoretic Peterson isomorphism.

\begin{definition}
For $d$-tuple $(\theta_1,\dots,\theta_d)$ of integers and $d$-tuple $(a_1,\dots,a_d)$ of non-negative integers, define
\[
D
\Array{\theta_1\,\theta_2\,\dots\,\theta_d\\a_1\,a_2\,\dots\,a_d}
:={[(1-\zeta)^{a_1}\zeta^{-\theta_1},(1-\zeta)^{a_2}\zeta^{-\theta_2},\dots,(1-\zeta)^{a_d}\zeta^{-\theta_d}]}.
\]
We also denote $D
\Array{\theta_1\,\theta_2\,\dots\,\theta_d\\0\,0\,\dots\,0}$ simply by 
$D(\theta_1,\theta_2,\dots,\theta_d).$
\end{definition}

\begin{lemma}\label{lemma:D}
$
D
\Array{\theta_1\,\theta_2\,\dots\,\theta_d\\a_1\,a_2\,\dots\,a_d}\in \Lambda_{(n)}
$
is uniquely determined by the following recursive formulas:
\begin{enumerate}
\item 
$
D
\Array{\dots\,\theta_i\,\dots\,\theta_j\,\dots\\\dots\,a_i\,\dots\,a_j\,\dots}
=
-
D
\Array{\dots\,\theta_j\,\dots\,\theta_i\,\dots\\\dots\,a_j\,\dots\,a_i\,\dots}
$,
\item 
$
D\Array{\dots\,\theta_i\,\dots\\\dots\,a_i\,\dots}
=
D\Array{\dots\,\theta_i-1\,\dots\\\dots\,a_i\,\dots}
+
D\Array{\dots\,\theta_i\,\dots\\\dots\,a_i+1\,\dots},
$
\item 
$
D\Array{\dots\,\theta\,\dots\\\dots\,n\,\dots}=0
$,
\item
$
D
\Array{0\,0\,\dots\,0\\n-\lambda_d-1\,n-\lambda_{d-1}-2\,\dots\,n-\lambda_1-d}
=s_{\lambda}$, 
where $s_\lambda$ is the Schur function.
\end{enumerate}
\end{lemma}

We have useful formulas such as
\begin{align}
&D\Array{d-1\,d-2\,\dots\,0\\0\,0\,\dots\,0}
=
D\Array{d-1\,d-2\,\dots\,0\\d-1\,d-2\,\dots\,0}
=g_{R_d},\\
&D\Array{d-1\,d-2\,\dots\,0\\n-\lambda_d-1\,n-\lambda_{d-1}-2\,\dots\,n-d-\lambda_1}
=g_{\lambda},\qquad (\lambda\subset R_d)\\
&
D\Array{\dots\,\theta\,\dots\\\dots\,a\,\dots}
=
\sum_{i}(-1)^{i}{p\choose i}
D\Array{\dots\,\theta+p\,\dots\\\dots\,a+i\,\dots}\label{eq:useful}.
\end{align}

\subsection{Lattice paths method}
Now we prove the following.
\begin{proposition}\label{prop:kappasigma} We have
$\sigma_d=\sum_{\mu\subset R_d }g_{\mu}$.
\end{proposition}
Since we have Lemma \ref{lemma:g}, it suffices to show the following.
 \begin{lemma}\label{lemma:appendix} We have
\begin{equation}\label{eq:toprove}
D\Array{d\,d-1\,\dots\,1\\d-1\,d-2\,\dots\,0}
=
\sum_{n>a_1>a_2>\dots>a_d\geq 0}
D\Array{d-1\,d-2\,\dots\,0\\a_1\,a_2\,\dots\,a_d}.
\end{equation}
\end{lemma}
\begin{proof}
Let
\[
K^{\theta}_{a,i}:={i-a+\theta-1\choose \theta-1}= 
\left(
\begin{array}{c}
\mbox{The number of weakly increasing sequences} \\
a\leq x_1\leq x_2\leq \dots\leq x_\theta\leq i
\end{array}
\right).
\]
By using (\ref{eq:useful}) repeatedly, we have
\begin{align*}
D\Array{\theta_1\,\dots\,\theta_d\\a_1\,\dots\,a_d}
&=\sum_{n>i_1>\cdots>i_d\geq 0}\,\sum_{\sigma\in S_d}\mathrm{sgn}(\sigma)\prod_{j=1}^{d}K^{\theta_j}_{a_j,i_{\sigma(j)}}
D\Array{0\,\dots\,0\\i_1\,\dots\,i_d}\\
&=\sum_{n>i_1>\cdots>i_d\geq 0}\det(K^{\theta_l}_{a_l,i_m})_{1\leq l,m\leq d}\cdot
D\Array{0\,\dots\,0\\i_1\,\dots\,i_d}.
\end{align*}
Especially,
\begin{align*}
&D\Array{d\,d-1\,\dots\,1\\d-1\,d-2\,\dots\,0}
=\sum_{n>i_1>\cdots>i_d\geq 0}\det(K^{d-l+1}_{d-l,i_m})_{1\leq l,m\leq d}\cdot
D\Array{0\,\dots\,0\\i_1\,\dots\,i_d},\\
&D\Array{d-1\,d-2\,\dots\,0\\a_1\,a_2\,\dots\,a_d}
=\sum_{n>i_1>\cdots>i_d\geq 0}\det(K^{d-l}_{a_l,i_m})_{1\leq l,m\leq d}\cdot
D\Array{0\,\dots\,0\\i_1\,\dots\,i_d}.
\end{align*}
Hence, it suffices to prove
\begin{equation}\label{eq2}
\det(K^{d-l+1}_{d-l,i_m})_{l,m}=\sum_{n> a_1>\dots>a_d\geq 0}\det(K^{d-l}_{a_l,i_m})_{l,m}
\end{equation}
for arbitrary $i_1,\dots,i_d$.

Now consider the plane lattice in Figure \ref{figure}.
Let $A_j$ ($j=1,\dots,d$) be the point with coordinates $(d-j+1,d-j)$ and $B_j$ be the point with $(0,i_j)$.
We immediately find that the number of shortest paths from $A_l$ to $B_m$ on the lattice is $K^{d-j+1}_{d-j,i_m}$.
By the Lindstr\"{o}m-Gessel-Viennot lemma~\cite{GesselViennot1985}, the determinant $\det(K^{d-l+1}_{d-l,i_m})_{l,m}$ equals to the number of shortest non-intersecting paths from the source set $\{A_1,\dots,A_d\}$ to the target set $\{B_1,\dots,B_d\}$.
(See Figure \ref{figure}).

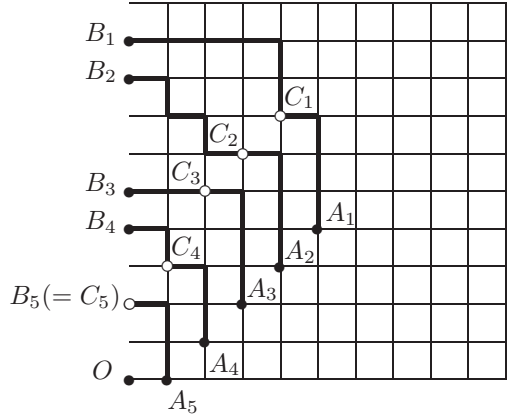
\begin{figure}[htbp]
\setlength\unitlength{0.5truemm}
\begin{center}
   \begin{picture}(100,100)
   \linethickness{0.3pt}
   \multiput(10,0)(10,0){10}{\line(0,1){100}}
   \multiput(0,0)(0,10){11}{\line(1,0){100}}
   \put(-2,-2){$\bullet$}
   \put(-10,0){$O$}
   \multiput(8,-2)(10,10){5}{$\bullet$}
   \put(10,-8){$A_5$}
   \put(21,2){$A_4$}
   \put(31,22){$A_3$}
   \put(41,32){$A_2$}
   \put(52,42){$A_1$}
   \put(-2,18){$\bullet$}
   \put(-2,38){$\bullet$}
   \put(-2,48){$\bullet$}
   \put(-2,78){$\bullet$}
   \put(-2,88){$\bullet$}
   \put(-32,20){$B_5(=C_5)$}
   \put(-12,40){$B_4$}
   \put(-12,50){$B_3$}   
   \put(-12,80){$B_2$}
   \put(-12,90){$B_1$}
   \linethickness{1.5pt}
   \put(0,90){\line(1,0){40}}
   \put(40,90){\line(0,-1){20}}
   \put(40,70){\line(1,0){10}}
   \put(50,70){\line(0,-1){30}}
   \put(0,80){\line(1,0){10}}
   \put(10,80){\line(0,-1){10}}
   \put(10,70){\line(1,0){10}}
   \put(20,70){\line(0,-1){10}}
   \put(20,60){\line(1,0){20}}
   \put(40,60){\line(0,-1){30}}
   \put(0,50){\line(1,0){30}}
   \put(30,50){\line(0,-1){30}}   
   \put(0,40){\line(1,0){10}}
   \put(10,40){\line(0,-1){10}}
   \put(10,30){\line(1,0){10}}
   \put(20,30){\line(0,-1){20}}
   \put(0,20){\line(1,0){10}}
   \put(10,20){\line(0,-1){20}}
   \newcommand{\whitedot}[2]
   {
   \put(#1,#2){\color{white}\circle*{3}} 
   \put(#1,#2){\circle{3}}   
   }
   \whitedot{0}{20}
   \whitedot{10}{30}
   \whitedot{20}{50}
   \whitedot{30}{60}
   \whitedot{40}{70}
   \put(11,33){$C_4$}
   \put(11,53){$C_3$}
   \put(21,63){$C_2$}
   \put(41,73){$C_1$}
   \end{picture}
\end{center}
\caption{
($d=5$). 
An example of non-intersecting paths from $\{A_1,\dots,A_5\}$ to $\{B_1,\dots,B_5\}$.
The plane lattice lacks the line on the $0$-th column. 
Here, $(i_1,i_2,i_3,i_4,i_5)=(9,8,5,4,2)$.
It corresponds to non-intersecting paths from $\{C_1,\dots,C_5\}$ to $\{B_1,\dots,B_5\}$ with $(a_1,a_2,a_3,a_4,a_5)=(7,6,5,3,2)$.
}
\label{figure}
\end{figure}

Let $C_j$ ($j=1,\dots,d$) be the point with coordinates $(d-j,a_j)$.
Similarly as above, the determinant $\det(K^{d-j}_{a_j,i_m})$ equals to the number of shortest non-intersecting paths from the source set $\{C_1,\dots,C_d\}$ to the target set $\{B_1,\dots,B_d\}$.
(See Figure \ref{figure}).

Denote 
\begin{gather*}
\mathfrak{X}=(\mbox{the set of shortest non-intersecting paths from $\{A_1,\dots,A_d\}$ to $\{B_1,\dots,B_d\}$}),\\
\mathfrak{Y}^{(a_1,\dots,a_d)}=
\left(
\begin{array}{c}
\mbox{The number of shortest non-intersecting paths}\\
\mbox{ from $\{C_1,\dots,C_d\}$ to $\{B_1,\dots,B_d\}$ with } C_j=(d-j,a_j)
\end{array}
\right).
\end{gather*}
There exists a natural one-to-one correspondence
\[
\bigcup_{n>a_1>\dots>a_d\geq 0}\mathfrak{Y}^{(a_1,\dots,a_d)}\to \mathfrak{X}
\]
which associates  
$$
(\mathcal{P}_1,\dots,\mathcal{P}_d)\in \mathfrak{Y}^{(a_1,\dots,a_d)},\qquad (\mathcal{P}_j\mbox{ is a path from $C_j$ to $B_j$})
$$
with the set of non-intersecting paths $(\mathcal{Q}_1,\dots,\mathcal{Q}_d)\in\mathfrak{X}$, where $\mathcal{Q}_j$ is the path which is obtained by adding to $\mathcal{P}_j$ a vertical segment from $A_j$ to $(d-j+1,a_j)$ and a horizontal edge from $(d-j+1,a_j)$ to $(d-j,a_j)$
(see Figure \ref{figure}).
Counting the cardinalities of these sets concludes (\ref{eq2}). \end{proof}

\setcounter{equation}{0}
\section{Quantum Grothendieck Polynomials of Grassmannian type} 
We fix an integer $d$ with $1\leq d\leq n$. Let $Gr_{d}(\mathbb{C}^n)$ denote the Grassmannian of $d$-dimensional subspaces of $\mathbb{C}^n.$
We will show how a quantization map for the Grassmannian $Gr_d(\mathbb{C}^n)$ can be interpreted 
in our context. Theorem \ref{thm:quantize} is the main result whose proof is given in the next section.
As an application, we obtain 
Theorem \ref{thm:Grass}.

\subsection{A result on $K$-theoretic Littlewood-Richardson rule}
For partitions $\lambda,\mu,\nu$, the $K$-{\it theoretic Littlewood-Richardson coefficients\/} is the integer $c_{\lambda,\mu}^\nu$ defined by 
$$
G_\lambda\cdot G_\mu=\sum_{\nu}(-1)^{|\nu|-|\lambda|-|\mu|}
c_{\lambda,\mu}^\nu\cdot G_\nu,
$$
where $\nu$ runs for all partitions. 
Buch's formula \cite{BuchLR} gives $c_{\lambda,\mu}^\nu$ as 
the number of set-valued tableaux which satisfy a certain property. 
Here we explain a rule of Ikeda-Shimazaki \cite{IkedaShimazaki}, which is equivalent to Buch's rule. 
Let $T$ be a set-valued tableau $T$.
The {\it column word\/} $cw(T)$ of $T$ is obtained by reading each column from top to bottom starting from the right most column to the left, where letters in the set filled in a box are read in the decreasing order. A word {\it builds\/} $\nu$ {\it on\/} $\lambda$
if $\nu$ is constructed by adding a box in the row associated to each entry of the word one by one while keeping the shape being a partition.
The coefficient $c_{\lambda,\mu}^\nu$ is the number of set-valued tableaux $T$ of shape $\mu$ such that $cw(T)$ builds $\nu$ from $\lambda$ (in \cite{IkedaShimazaki}, such $T$ is called $\lambda$-good tableaux of shape $\mu$ of content $\nu-\lambda$).

The following is stated without proof in \cite{BuchCombK2005} (proof of Corollary 1).
We include a proof for completeness.
\begin{proposition}\label{prop:LR} Let $\lambda$ be a partition such that 
$\lambda\subset R_d.$ For any partition $\mu$,
we have \begin{equation}
c_{\lambda,\mu}^{R_d}=\delta_{\lambda^{\!\vee}\!, \mu}.\label{eq:LR}
\end{equation} 
\end{proposition}
\begin{proof}(T. Matsumura) 
Let $T$ be a set-valued tableau such that $cw(T)$ builds $R_d$ on $\lambda.$ 
We prove that $T$ is the ordinary semistandard tableau $T_0$ such that 
$i$th column of $T_0$ is filled with $d,d-1,\ldots,r_i+1$ from the bottom to top, 
where $r_i$ is the number of boxes in $(n-d+i)$th column of $\lambda$ (see Example \ref{example:tableau} below). 
This in particular implies that the shape of $T_0$ is $\lambda^{\!\vee}.$

We proceed by induction on $i$ to prove that the first $i$ columns of $T$ coincide with those of $T_0.$ Let us consider the base case $i=1.$
Observe that $R_d$ has exactly one southeast corner, whose row index is $d$, and
consequently, the last letter of the column word $cw(T)$ should be $d$, which is 
the minimum entry of the bottom box of the first column of $\mu.$
The second last entry of $cw(T)$ is either $d$ or $d-1$. If it is $d$, then it cannot be in the 
the first column of $\mu$, because each column of $T$ contains at most one $d$. 
It follows that the first column of $T$ consists of one box filled with only $d$,
and hence $\mu$ consists of one row. 
Then $T$ cannot contain $d-1$.
It means that the right most column of $\lambda$ has $d-1$ boxes. 
Thus this case is done. 
Suppose the second last entry of $cw(T)$ is $d-1$. We claim that $\mu$ has at least two rows. If $\mu$, on the contrary, has only one row, then $T$ cannot contain $d-1$ ( $d$ is 
the minimum entry of the bottom box of the first column of $\mu$).
Next we claim that
this $d-1$
is the only entry in the box above the bottom box of the first column of $\mu$ (there should be $d-1$ in the box above the bottom, and there cannot be other entry in the box, since otherwise other smaller entry becomes the second last entry of $cw(T)$). 
If $\mu$ has exactly two rows, there are no $d-2$ in $T$ and the right most column of $\lambda$ has $d-2$ boxes. Then the case is done.
We can proceed in this way to see that $T$ coincides with $T_0$ at the first column.

Assume that the first $i-1$ columns of $T$ coincide with those of $T_0.$
We consider the column word of the part $T_{\geq i}$ consisting of
the last $(n-d-i+1)$ columns of $T$.
Let $\overline{\lambda}$ be the partition
obtained by removing the boxes of $R_d$ corresponding to  
the row numbers of the first $(i-1)$ columns of $T$.
The last letter of $cw(T_{\geq i})$ is $d$, because
$i$th column of $\overline{\lambda}$ has $d$ boxes,
and the bottom box of this column is the only box of $\overline{\lambda}$ which is
a southeast corner and does not belong to $\lambda$.
Then 
by the same argument of the case $i=1$, 
the entry of the bottom box of the first column of $T_{\geq i}$, the column $i$ of $T$, is $\{d\}.$
Moreover, by the same argument of the case $i=1$,
we know that $i$-th column of $T$ coincides with the one of $T_0.$
Hence the induction completes. 
\end{proof}
\begin{example}\label{example:tableau} Let $\lambda=(6,5,2,1)$ and $d=4, n=10$.
Then $\lambda^\vee=(5,4,1)$ and $T_0$ is given as follows:
$$\lambda= \Tableau{&&&&&\\&&&&\\&\\{}},\quad
T_0=\Tableau{2& 3 & 3 & 3 & 4\\ 3 & 4 & 4& 4 \\ 4}.$$
\end{example}

\subsection{Grothendieck polynomials}\label{ssec:GroP}
Grothendieck polynomials were introduced by Lascoux and Sch\"utzenberger \cite{LascouxSchutzenberger1982Groth} as polynomial representatives for the classes of structure sheaves of Schubert vatieties in $K(Fl_n)$ (cf. (\ref{eq:presenK}) below).
For each $1\leq i\leq n-1$,
we define
the {\it isobaric divided difference operator\/} given by
$$
\pi_i f=\frac{(1-x_{i+1})f-(1-x_i)s_if}{x_i-x_{i+1}}\quad (
f
\in \mathbb{Z}[x_1,\ldots,x_n])
$$
where the simple reflection $s_i=(i,i+1)$ acts by exchanging $x_i$ and $x_{i+1}.$
If $w_0=(n,n-1,\ldots,1)$ is the longest permutation in $S_n$ we set 
$$
\mathfrak{G}_{w_0}=x_1^{n-1}x_2^{n-2}\cdots x_{n-1}.
$$
There exist a unique family $\{\mathfrak{G}_w(x)\;|w\in S_n\}$ of polynomials 
such that 
$$
\pi_i\mathfrak{G}_w=\begin{cases}
\mathfrak{G}_{ws_i} & \mbox{if}\;\ell(ws_i)=\ell(w)-1\\
\mathfrak{G}_w &\mbox{if}\; \ell(ws_i)=\ell(w)+1.
\end{cases}
$$
\subsection{Quantization map of $K(Gr_d(\mathbb{C}^n))$ and $K$-Peterson isomorphism}
Let $\hat{\Lambda}$ denote the $\mathbb{C}$-span of the stable 
Grothendieck polynomials. This is a completion of the ring $\Lambda$.
It was shown in \cite{BuchLR} that $\hat{\Lambda}$, denoted as $\Gamma$ in the paper, is closed under multiplication.
Denote by $J_d$ the ideal of $\hat{\Lambda}$ defined as
\[
J_d:=\{f\in \hat{\Lambda}\,\vert\,f^\perp\cdot g_{R_d}=0 \}.
\]

\begin{proposition}\label{prop:KGr}
There is a canonical isomorphism
$
\textstyle K(Gr_d(\CC^n))\simeq \hat{\Lambda}/J_d.
$
\end{proposition}
\begin{proof}
The linear span of $G_\mu$'s such that $\mu\not\subset R_d$
is  an ideal of $\hat \Lambda$ and $\hat \Lambda/I_d$ is isomorphic $K(Gr_d(\CC^n))$ (Buch \cite{BuchLR}, Theorem 8.1).
From Proposition \ref{prop:LR}, we see that $J_d$ coincides with $I_d.$
\end{proof}

Via the isomorphism in Proposition \ref{prop:KGr},
we define the $\CC$-linear map 
$$
\widehat{Q}_d:K(Gr_d(\CC^n))\to \mathcal{QK}(Fl_n),\qquad f\!\!\!\mod{J_d}\mapsto \Phi_n^{-1}
\left(
{(f^\perp\cdot g_{R_d})}/{g_{R_d}}
\right).
$$ 

\subsection{Quantization map of Lenart--Maeno}
For $1\leq m\leq n$, define a polynomial $F^{(m)}_i\in \CC[x,Q]$ (cf. \cite[\S 3]{lenart2006quantum}) by
$$
F^{(m)}_i=\sum_{\substack{I\subset \{1,\ldots,m\}\\\;\# I=i}}\prod_{j\in I}(1-x_j)\prod_{j\in I,\;j+1\notin I }(1-Q_j),
$$
where $Q_n:=0$.
Note that $F^{(n)}_i$ is nothing but $F_i$ with $z_j=1-x_j.$

In \cite{lenart2006quantum}, Lenart and Maeno introduced the {\it quantization map} $\widehat{Q}$ and defined the {\it quantum Grothendieck polynomials} by using it.
Let $e_i$ be the $i$th elementary symmetric polynomial.
Let 
$f_i^{(j)}=e_i(1-x_1,\dots,1-x_j)$ for $1\leq i,j\leq n$. The following presentation is well-known (recall that $x_i$ is the $K$-theoretic first Chern class $c_1(\mathcal{L}_i^\vee):=1-[\mathcal{L}_i]$ of the dual of the tautological line bundle $\mathcal{L}_i$):
\begin{equation}
K(Fl_n)\simeq \mathbb{C}[x_1,\ldots,x_n]/\langle e_i(x_1,\ldots,x_n)|1\leq i\leq n\rangle.
\label{eq:presenK}
\end{equation}
Note that the ideal $\langle e_i(x_1,\ldots,x_n)|1\leq i\leq n\rangle$
is also generated by $f_i^{(n)}-\binom{n}{i}\;(1\leq i\leq n)$.
Let $L_n$ be the $\CC$-vector subspace of $\CC[x_1,\dots,x_n]$ generated by the elements
\begin{equation}
f^{(1)}_{i_1}f^{(2)}_{i_2}\cdots f^{(n-1)}_{i_{n-1}}\qquad (0\leq i_j\leq j).\label{eq:Fbasis}
\end{equation}
There exists a canonical isomorphism $K(Fl_n)\simeq L_n$ of $\CC$-vector spaces (\cite{lenart2006quantum}). 
\begin{definition}[\cite{lenart2006quantum}]\label{defn:hatQ}
The quantization map $\widehat{Q}:K(Fl_n)\to \mathcal{QK}(Fl_n)$ is the $\CC$-linear map defined by 
\begin{equation}
\widehat{Q}(f^{(1)}_{i_1}f^{(2)}_{i_2}\cdots f^{(n-1)}_{i_{n-1}}):=F^{(1)}_{i_1}F^{(2)}_{i_2}\cdots F^{(n-1)}_{i_{n-1}}\qquad (0\leq i_j\leq j).\label{eq:hatQ}
\end{equation}
\end{definition}
\begin{remark} The definition of $\widehat Q$ given in \cite{lenart2006quantum} is 
equivalent to Definition \ref{defn:hatQ} (see \cite[Proposition 3.16]{lenart2006quantum}).
\end{remark}

The quantum Grothendieck polynomial $\mathfrak{G}_w^Q$, for $w\in S_n$, is 
defined as 
$$
\mathfrak{G}_w^Q=\widehat Q(\mathfrak{G}_w).
$$
For $f\in \hat \Lambda$, let $f(x_1,\dots,x_d)$ denotes the polynomial
by setting $x_i=0$ for $i>d$ in the symmetric function $f\in \hat\Lambda$.
Let  $\pi: Fl_n\rightarrow Gr_d(\CC^n)$ be the projection sending $V_\bullet$ to $V_d$.
The induced morphism $\pi^* : K(Gr_d(\CC^n))\simeq \hat{\Lambda}/J_d\hookrightarrow K(Fl_n)$
is given by $f\!\!\mod{J_d}\mapsto f(x_1,\dots,x_d)$.

The main statement of this section is the following. The proof is given in \S \ref{sec:ProofGr}.

\begin{theorem}\label{thm:quantize}
The following diagram commutes
\[
\xymatrix
{\hat{\Lambda}/J_d \simeq
K(Gr_d(\CC^n)) \ar[rd]_{\widehat{Q}_d}\ar@{^{(}->}^{\pi^*}[rr]& & K(Fl_n)\ar[dl]^{\widehat{Q}}\\
 & \mathcal{QK}(Fl_n)&
}.
\]
\end{theorem}

\begin{corollary}\label{cor:Gr} Let $\lambda\subset R_d.$ We have
$\widehat{Q}_d(G_\lambda\;\mathrm{mod}\; J_d)=
\mathfrak{G}_{w_{\lambda,d}}^Q.$
\end{corollary}
\begin{proof}
It is known that $G_\lambda(x_1,\ldots,x_d)=\mathfrak{G}_{\lambda,d}$ (Buch \cite{BuchLR}, \S 8), so we have
$$
\widehat Q_d (G_\lambda\;\mathrm{mod}\; J_d)
=\widehat Q(G_\lambda(x_1,\ldots,x_d))
=\widehat Q(\mathfrak{G}_{\lambda,d})
=\mathfrak{G}_{\lambda,d}^Q.
$$
\end{proof}

Now we prove Theorem \ref{thm:Grass}.
\begin{proof}
Corollary \ref{cor:Gr} is equivalent to
$$
\Phi_n(\mathfrak{G}_{w_{\lambda,d}}^Q)=\frac{G_\lambda^\perp\cdot g_{R_d}}{g_{R_d}}.
$$
So we need to show
$G_\lambda^\perp\cdot g_{R_d}=g_{\lambda^\vee},$
which is equivalent to Proposition \ref{prop:LR}.
\end{proof}

\setcounter{equation}{0}
\section{Proof of Theorem \ref{thm:quantize}}\label{sec:ProofGr}
\subsection{Outline of the proof}
Let $\lambda$ be partition such that $\lambda\subset R_d.$ 
We define a quantized version of Schur polynomial 
$$
{S}^{Q}_{\lambda,d}:=\det(F^{(d+j-1)}_{\lambda_i'-i+j})_{i,j=1}^{\ell(\lambda')}.
$$
Note that the polynomial ${S}^{Q}_{\lambda,d}$ is an element in 
$\mathbb{Z}[Q][x_1,\ldots,x_d]^{S_d}.$
\begin{proposition} We have
\begin{equation}
\widehat{Q}(s_\lambda(1-x_1,\dots,1-x_d))=S_{\lambda,d}^Q.\label{eq:hat Q s}
\end{equation}
\end{proposition}
\begin{proof} Let $e_i^{(j)}=e_i(z_1,\ldots,z_j).$ 
We know the dual Jacobi-Trudi formula $$s_\lambda(z_1,\dots,z_d)=\det(e_{\lambda_i'-i+j}^{(d)})_{i,j=1}^{\ell(\lambda')}.$$
Since $e_i^{(j)}=e_i^{(j-1)}+z_{j}e_{i-1}^{(j-1)},$
it is easy to show $s_\lambda(z_1,\dots,z_d)=\det(e_{\lambda_i'-i+j}^{(d+j-1)})_{i,j=1}^{\ell(\lambda')}$.
Now by substituting $z_i=1-x_i$ to this equality, we have
\begin{equation}
s_\lambda(1-x_1,\dots,1-x_d)=\det(f_{\lambda_i'-i+j}^{(d+j-1)})_{i,j=1}^{\ell(\lambda')}.\label{eq:flags}
\end{equation}
One observes that each term of the expansion on the right hand side of (\ref{eq:flags}) 
is of the from (\ref{eq:Fbasis}). Then (\ref{eq:hat Q s})
follows from the definition of $\widehat{Q}$. 
\end{proof}

Let $p_i\in \Lambda$ be the $i$th power sum symmetric function (\cite{macdonald1998symmetric}).
We consider the ring homomorphism $\kappa_d: \Lambda\rightarrow \Lambda$ given by 
\begin{equation}\label{eq:tildep}
\kappa_d(p_i):=d-{i\choose 1}p_1+{i\choose 2}p_2-\dots+(-1)^{i}{i\choose i}p_{i}.
\end{equation}
Recall that each element of $\Lambda$ is a symmetric formal power series 
in $x=(x_1,x_2,\ldots)$. $\kappa_d$ is given by $x_i\mapsto 1-x_i\;(1\leq i\leq d)$, 
$x_j\mapsto x_j\;(j>d).$
Thus obviously $\kappa_d$  is an involution. We will show the following in the sequel of this section.
\begin{proposition}\label{prop:PhiF} We have
\begin{equation}
\Phi_n({S}^{Q}_{\lambda,d})=\frac{\kappa_d(s_\lambda)^\perp\cdot g_{R_d}}{g_{R_d}}.\label{eq:PhiF}
\end{equation}\end{proposition}

The equation (\ref{eq:PhiF}) is equivalent to
\begin{equation}
\widehat{Q}_d(\kappa_d(s_{\lambda}) \mod J_d)={S}^{Q}_{\lambda,d}.
\end{equation}
On the other hand, the element 
$
\kappa_d(s_\lambda)\mod J_d
$ of $\hat \Lambda/J_d$ is mapped to  
$
s_\lambda(1-x_1,\dots,1-x_d)\in L_n\simeq K(Fl_n).
$
Thus we have
$$
\widehat Q\left(\pi^*(\kappa_d(s_{\lambda}) \mod J_d)\right)=
\widehat Q\left(s_\lambda(1-x_1,\dots,1-x_d)\right)=S_{\lambda,d}^Q.
$$ 
Since 
$\kappa_d(s_{\lambda}) \mod J_d\;(\lambda\subset R_d)$
form a basis of $\hat \Lambda/J_d,$
Theorem \ref{thm:quantize} holds.
\subsection{$\Phi_n({S}^{Q}_{\lambda,d})$ as a ratio of determinants}
In this subsection we prove the following proposition.

\begin{proposition}\label{prop:determinant2}Let $\lambda$ be a partition contained in $R_d$. We define  
the increasing sequence $(i_1,\cdots,i_d)$ by setting
\begin{equation}
i_a=\lambda_{d+1-a}+a\quad(a=1,\ldots,d). \label{eq:ilambda}
\end{equation}
Then we have
\begin{equation}
\renewcommand\arraystretch{0}
\Phi_n({S}^{Q}_{\lambda,d})=
\frac
{
  D(d-i_1,d-i_2,\ldots,d-i_d)
}
{
  D(d-1,d-2,\dots,0)
}.\label{eq:determinant2}
\end{equation}
\end{proposition}

\begin{remark}
Almost all the statements and proofs of this subsection make sense for arbitrary values of $\gamma$, which will be relevant when we discuss equivariant case (cf. \cite{lam2011double}). 
\end{remark}

\begin{lemma}\label{lemma:RtoF}
Let $L\in Z_\uni^\circ$, and $U$ be the matrix constructed in Proposition \ref{prop:RUexist}. 
The components $u_{ij}\;(i>j)$ of $U$ is equal to $(-1)^{j-1}F_{i-j}^{(i-1)}$. 
\end{lemma}
\begin{proof} 
We have $UC_\uni U^{-1}=L.$
Let us compare both hand sides of principal minors of 
$\zeta \cdot 1-UC_\uni U^{-1}=\zeta \cdot 1-L.$
In view of the facts that
 $L$ satisfies $(\mathrm{Z}_1)$,
and $U\in \pmb{N}_{\!-}\varepsilon,$ one can show that 
$$\xi_{1,\ldots,i}^{1,\ldots,i}(\zeta \cdot 1-L)
=\zeta^i+(-1)^i\sum_{j=1}^i u_{i+1,i-j+1}\zeta^{i-j}\quad (1\leq i\leq n-1),
$$
On the other hand, we have 
$\xi_{1,\ldots,i}^{1,\ldots,i}(\zeta \cdot 1-L)=\zeta^i+\sum_{j=1}^i(-1)^jF_j^{(i)}\zeta^{i-j}.$
Thus the lemma is proved.
\end{proof}

\begin{lemma}[\cite{Noumi9th}]\label{lem:GaussMinor}
Let $\lambda$ and $ (i_1,\ldots,i_d)$ be as Proposition \ref{prop:determinant2}. 
Let $s=\ell(\lambda')$, where $\lambda'$ is the conjugate partition of $\lambda$.
We define the increasing sequence $(j_1,\ldots,j_{s})$ by the 
condition
$$\{i_1,\ldots,i_d\}\cup\{j_1,\ldots,j_{s}\}
=\{1,2,\ldots,d+s\}.$$
Suppose a matrix $X$
is decomposed as $X=Y\cdot N^{-1},\;
Y\in \pmb{B}_{\!-},\;
N\in \pmb{N}.$
Then we have the expression 
\begin{equation}
\xi^{d+1,\ldots,d+s}_{j_1,\ldots,j_{s}}(N)=(-1)^{|\lambda|}\cdot\frac{\xi_{1,\ldots,d}^{i_1,\ldots,i_d}(X)}{\xi_{1,\ldots,d}^{1,\ldots,d}(X)}.\label{eq:GaussMinor}
\end{equation}
\end{lemma}
\begin{proof}See \cite{Noumi9th}, Theorem 1.1 and its proof.
\end{proof}

Now we prove Proposition \ref{prop:determinant2}.
\begin{proof}
Let $\{j_1,\ldots,j_s\}$ as in Lemma \ref{lem:GaussMinor}.
From Lemma \ref{lemma:RtoF}, we see that  
\begin{equation}
{S}^{Q}_{\lambda,d}=
\xi_{d+1,\ldots,d+s}^{j_1,\ldots,j_{s}}(U\varepsilon).\label{eq:S=xi}
\end{equation}

We apply Lemma \ref{lem:GaussMinor} as follows.
Recall that we have the decomposition $\varphi(C_\uni)=U^{-1}R$
with $R\in \pmb{B}\sigma$ and $U\in \pmb{N}_{\!-} \varepsilon.$
Let $Y={}^t(R \sigma^{-1}),\;N={}^t(U\varepsilon),\;X=\sigma\cdot {}^t\varphi(C_\uni)\,\varepsilon.$
If we choose $\pmb{c}: \mathscr{O}_\uni\rightarrow \mathbb{C}^n$ as $\sum_{i=0}^{n-1}\alpha_i\zeta^i \mod (\zeta-1)^n
\mapsto {}^t(\alpha_{n-1},\alpha_0,\ldots,\alpha_{n-2}),$ then
we have 
$$
\sigma\cdot{}^t\varphi(C_\uni)=(\pmb{b}_0,\pmb{b}_1,\ldots,\pmb{b}_{n-1}).
$$
Now we have
\begin{align}
\xi_{1,\ldots,d}^{i_1,\ldots,i_d}(X)
&=(-1)^{\sum_{a=1}^d(i_a-1)}|\pmb{b}_{i_1-1},\ldots,\pmb{b}_{i_d-1},\pmb{a}_{d-1},\ldots \pmb{a}_{n-2}|\nonumber\\
&=(-1)^{\sum_{a=1}^d(i_a-1)}|\pmb{c}(\zeta^{i_1-1}\varphi),\ldots,\pmb{c}(\zeta^{i_d-1}\varphi),\pmb{c}(\zeta^{d-1}),\ldots \pmb{c}(\zeta^{n-2})|\nonumber\\
&=(-1)^{\sum_{a=1}^d(i_a-1)}|\pmb{c}(\zeta^{i_1-d}\varphi),\ldots,\pmb{c}(\zeta^{i_d-d}\varphi),\pmb{c}(1),\ldots ,\pmb{c}(\zeta^{n-d-1})|\nonumber\\
&=(-1)^{|\lambda|+d(d-1)/2}D(d-i_1,\ldots,d-i_d)c_0^d,\label{eq:xi=D}
\end{align}
where we used $\sum_{a=1}^di_a=|\lambda|+d(d+1)/2.$ 
Since we have
$$\xi_{d+1,\ldots,d+s}^{j_1,\ldots,j_{s}}(U\varepsilon)
=\xi^{d+1,\ldots,d+s}_{j_1,\ldots,j_{s}}(N),
$$
formula (\ref{eq:determinant2}) follows from (\ref{eq:GaussMinor}), (\ref{eq:S=xi}), and (\ref{eq:xi=D}).
\end{proof}

\begin{remark}\label{rem:GrassPerm} The $d$-Grassmannian permutation $w_{\lambda,d}\in S_n$ is given by 
$w_{\lambda,d}(a)=i_a\;(1\leq a\leq d)$ and $w_{\lambda,d}(d+a)=j_a\;(1\leq a\leq s)$, 
and $w_{\lambda,d}(a)=a\;(a>d+s).$
\end{remark}

\subsection{Calculation of $\kappa_d(s_\lambda)^\perp\cdot g_{R_d}$}

In view of Proposition \ref{prop:determinant2}, Proposition \ref{prop:PhiF}
is reduced to the following.
\begin{proposition}\label{prop:numer}
We have
\begin{equation}\label{eq:concl}
\kappa_d(s_\lambda)^\perp\cdot g_{R_d}=D(d-i_1,d-i_2,\ldots,d-i_d).
\end{equation} 
\end{proposition}
In the rest of the paper, we will concentrate on the proof of Proposition \ref{prop:numer}.

\subsubsection{Actions of $\kappa_d(p_i)^\perp$}
\begin{lemma}\label{lemma:4.2} We have 
\begin{equation}
\kappa_d(p_i)^\perp \cdot
  D\Array{\theta_1\,\dots\,\theta_d\\a_1\,\dots\,a_d}
=
\displaystyle\sum_{j=1}^{d}
  D\Array{\theta_1\,\dots\,\theta_j-i\,\dots\,\theta_d\\a_1\,\dots\,a_j\,\dots\,a_d}.
  \label{eq:p[k]perp}
\end{equation}
\end{lemma}
\begin{proof}
We first show
\begin{equation}
p_i^\perp \cdot
  D\Array{\theta_1\,\dots\,\theta_d\\a_1\,\dots\,a_d}
=\displaystyle\sum_{j=1}^{d}
  D\Array{\theta_1\,\dots\,\theta_j\,\dots\,\theta_d\\a_1\,\dots\,a_j+i\,\dots\,a_d}.
  \label{eq:pperp}
\end{equation}
As $p_i^\perp h_j=h_{j-i}$ (\cite{macdonald1998symmetric}, Chap. I, 5, Example 3), the action of $p_i^\perp$ on the column vector 
$$
\pmb{c}((1-\zeta)^{c}\zeta^{-\theta}\varphi)=c_0\cdot \pmb{c}
\left(
(1-\zeta)^{c}\zeta^{-\theta}(\sum_{i=0}^{n-1}h_i(1-\zeta)^i
)
\right)
$$
is expressed as 
$$
p_i^\perp\cdot \pmb{c}((1-\zeta)^{c}\zeta^{-\theta}\varphi)=\pmb{c}((1-\zeta)^{a+i}\zeta^{-\theta}\varphi).
$$
This relation and the `Leibniz rule' $p_i^\perp (fg)=(p_i^\perp f)g+f(p_i^\perp g)$ imply the desired equation.
(\ref{eq:p[k]perp})
 is obtained from (\ref{eq:pperp}) by using (\ref{eq:useful}) as follows:
\begin{align*}
\kappa_d(p_i)^\perp\cdot D\Array{\theta_1\,\dots\,\theta_d\\a_1\,\dots\,a_d}
&=
\displaystyle\sum_{j=1}^{d}
\sum_{m=0}^{i}(-1)^m{i\choose m}
  D\Array{\theta_1\,\dots\,\theta_j\,\dots\,\theta_d\\a_1\,\dots\,a_j+m\,\dots\,a_d}\\
&
=
\displaystyle\sum_{j=1}^{d}
  D\Array{\theta_1\,\dots\,\theta_j-i\,\dots\,\theta_d\\a_1\,\dots\,a_j\,\dots\,a_d}.
\end{align*}
\end{proof}

\subsubsection{Boson-Fermion correspondence}
To prove Proposition \ref{prop:numer}, we use the {\it Boson-Fermion correspondence}.
Here we review some basic facts about it without proof.
For details, see \cite{kac1988bombay,miwa2000solitons}.

Let $\mathscr{M}:=\{M=(m_0,m_1,m_2,\ldots)\,\vert\,m_0>m_1>\cdots ,\ \ m_j=-j\ (j\gg 1)\}$.
Let $v_m$ be infinitely many linearly independent vectors indexed by $m\in \mathbb{Z}.$
Let 
$$
\mathscr{F}=\bigoplus_{M\in \mathscr{M}}\CC v_M,\quad 
v_M:=v_{m_0}\wedge v_{m_1}\wedge\cdots
$$
be the {\it Fermion-Fock space}.
The vector $\Omega:=v_0\wedge v_{-1}\wedge v_{-2}\wedge \cdots$ is called the {\it vacuum vector} of $\mathscr{F}$.
For $m\in \mathbb{Z},\;m\neq 0$, define $\alpha_{m}\in \mathrm{End}_{\mathbb{C}}\mathscr{F}$ by the formula:
\[
\alpha_m(v_{m_0}\wedge v_{m_1}\wedge\cdots)=\sum_{j=0}^\infty
v_{m_0}\wedge\dots\wedge v_{m_{j-1}}\wedge v_{m_{j}-m}\wedge v_{m_{j+1}}\wedge\cdots.
\]
Then we have the {\it Heisenberg relation\/}:
$$
[\alpha_m,\alpha_n]=m\delta_{m+n,0}.
$$
There uniquely exists a linear isomorphism $\phi:\mathscr{F}\to \Lambda$ with the following properties:
\[
\phi(\Omega)=1,\quad
\phi(\alpha_{-m}v)=p_m\phi(v),\quad \phi(\alpha_mv)=p_m^\perp\phi(v),\quad m\geq 1,\quad v\in \mathscr{F}.
\]
\begin{proposition}$($see~\cite[\S 9.3]{miwa2000solitons},~\cite[\S 6]{kac1988bombay}$)$. \label{prop:facts} We have
$
\phi(v_{m_0}\wedge v_{m_1}\wedge\cdots)=s_\lambda
$,
where $\lambda=(m_0,m_1+1,m_2+2,\dots)$ considered as a partition.
\end{proposition}

\subsubsection{Proof of Proposition \ref{prop:numer}.}
\begin{proof}

Consider the subset $\mathscr{M}_d\subset \mathscr{M}$ which is defined by
\[
\mathscr{M}_d=\{(m_j)\in \mathscr{M}\,\vert\,m_j=-j \;(j\geq d)\}
\]
and the subspaces 
$\mathscr{F}_{d}, \mathscr{W}_{d}$ of $ \mathscr{F}$ defined by
\begin{align*}
\mathscr{F}_d=\bigoplus_{M\in \mathscr{M}_d}\CC\cdot v_{M},\quad
\mathscr{W}_d=\bigoplus_{M\in \mathscr{M}\setminus \mathscr{M}_d}\CC\cdot v_{M}.
\end{align*}
The space $\mathscr{F}$ decomposes as
$\mathscr{F}=\mathscr{F}_d\oplus \mathscr{W}_d$.
Let $v\mapsto \overline{v}, \mathscr{F}\rightarrow \mathscr{F}_d$ be the projection.

Let $\iota:\mathscr{F}_d\to \Lambda$ be the linear map
\[
v_{m_0}\wedge\cdots \wedge v_{m_{d-1}}\wedge v_{-d}\wedge v_{-d-1}\wedge\cdots \mapsto D(-m_{d-1},\dots, -m_0).
\]
Note that $\iota(\Omega)=D(d-1,d-2,\ldots,0)=g_{R_d}.$ 
Lemma \ref{lemma:4.2} can be rewritten as 
\begin{equation}
\kappa_d (p_i)^\perp\cdot\iota(v)=\iota\left(\overline{\alpha_{-i}(v)}\right),\qquad v\in\mathscr{F}_d.
\end{equation}
For $v\in \mathscr{F}_d$ define 
$\hat s_\lambda(v)=\overline{
s_\lambda\cdot 
v}\in \mathscr{F}_d$, where the Schur function
$s_\lambda $ acts on $\mathscr{F}$ via the identification $p_m\mapsto \alpha_{-m}\;(m\geq 1).$
Accordingly we have
\begin{equation}\label{eq:key}
\kappa_d(s_\lambda)^\perp\cdot \iota(v)=\iota
\left( \hat s_\lambda(v)\right),\qquad v\in\mathscr{F}_d.
\end{equation}
For a partition $\lambda$ with $\ell(\lambda)\leq d$
we have the equation
\begin{align*}
&\hat s_\lambda\cdot \Omega\\
&=\phi^{-1}(s_\lambda\cdot 1)\\
&
=v_{\lambda_1}\wedge v_{\lambda_2-1}\wedge \cdots
\wedge v_{\lambda_d-d+1}\wedge v_{-d}\wedge v_{-d-1}\wedge \cdots
\quad (\mbox{Prop.}\; \ref{prop:facts})
\\
&
=v_{i_d-d}\wedge v_{i_{d-1}-d}\wedge \cdots
\wedge v_{i_1-d}\wedge v_{-d}\wedge v_{-d-1}\wedge \cdots \quad (\mbox{by} \quad (\ref{eq:ilambda})),
\end{align*}
and hence by substituting $v=\Omega$ to (\ref{eq:key}) we have
$$
\kappa_d(s_\lambda)^\perp \cdot g_{R_d}
=\kappa_d(s_\lambda)^\perp\cdot \iota(\Omega)
=\iota(\hat s_\lambda\cdot \Omega)
=D(i_1-d,\ldots,i_d-d).
$$
\end{proof}

\section{Discussion of Conjecture \ref{tildeg}}
The aim of this section to explain some details about Conjecture \ref{tildeg}
for the image of the quantum Grothendieck polynomials
by $\Phi$ associated with arbitrary permutations.

\subsection{$\lambda$-map}\label{ssec:lambda}

Recall that we set $k=n-1$. Let $ \mathcal{B}_{k}$ denote the set of $k$-bounded partitions. 
We recall the definition of a map 
$\lambda:S_n\rightarrow \mathcal{B}_{k}$ 
due to Lam and Shimozono \cite[\S 6]{lamshimo2010toda}.
For $0\leq i\leq n-2$, let $c_i$ denote the cyclic permutation $(i+1,i+2,\cdots,n)$, and
$C$ denote the cyclic subgroup generated by $c_0=(12\cdots n)$. 
For $w\in S_n$, let $\tilde{w}$ be the unique element in the coset $C\cdot w$ such that
$\tilde{w}(1)=1.$ There is a unique sequence $(m_1,\ldots,m_{n-2})$ of non-negative integers such that 
\begin{equation}
\tilde{w}=c_1^{m_1} c_2^{m_2}\cdots c_{n-2}^{m_{n-2}}\quad 
(0\leq m_i\leq k-i).
\end{equation}
Define $\lambda(w)=(1^{m_1}2^{m_2}\cdots (n-2)^{m_{n-2}})$, the partition
whose multiplicity of $i\;(1\leq i\leq n-2)$ is $m_i.$ 


\begin{lemma}\label{lem:kconj}
Let $\mu$ be a partition contained in $R_d$. We have
$
\lambda(w_{\mu,d})^{\omega_k}=\mu^\vee.
$  
\end{lemma}
\begin{proof} We write $w=w_{\mu,d}.$
We first assume that $w(1)=1$, that is $\mu_d=0$. It is straightforward to see
$w=c_1^{m_1}c_2^{m_2}\cdots c_{d-1}^{m_{d-1}}c_d^{m_d}$
with
$
m_i=w(i+1)-w(i)-1\;(1\leq i\leq d-1),$ and 
$m_d=n-w(d).
$
Note that the Young diagram of $\lambda(w)=(1^{m_1}\cdots d^{m_d})$ is contained in 
the conjugate of $R_d$, that is $R_{n-d}.$ 
If we consider the complement $\mu^c=R_d\setminus \mu$ of $\mu$ in the rectangle $R_d$, 
the number of columns in $\mu^c$ having $i$ boxes is $m_i.$ 
The diagram of $\mu^\vee$ is obtained from $\mu^c$ by a rotation of 
$180$ degrees. Now the conjugate of $\mu^\vee$ is nothing but $\lambda(w)$.
Since the $k$-conjugate of a partition contained in $R_d$ is the ordinary conjugate of it (\cite[Remark 10]{LapointeMorse2005JCTA}), the lemma follows in this case. If $w(i)>1$, then $\tilde w=c_0^{-w(1)+1}w$ is $(d+l)$-Grassmannian with some $l$ such that $1\leq l\leq k-d$. Let $\tilde \mu
\subset R_{d+l}$ be the corresponding partition. One can check that $\tilde\mu^c=R_{d+l}\setminus \tilde \mu$ has the same shape as 
$\mu^c=R_d\setminus \mu,$ and hence the proof is reduced to the case when $w(1)=1.$
\end{proof}

A $k$-bounded partition $\mu$ 
is $k$-{\it irreducible\/} if there is no $k$-rectangle $R_d$ such that $\mu
=R_d\cup \nu\;(\nu\in \mathcal{B}_k)$.
This is equivalent to the inequalities  
$0\leq m_i\leq k-i\;(1\leq i\leq n-2),$ where $m_i$ is the multiplicity of $i$ in $\mu.$ 
Accordingly the image of the map $\lambda$
is contained in the set $\mathcal{B}_k^*$ of all irreducible $k$-bounded partitions. 
In fact, one can easily see that $\mathcal{B}_k^*$  coincides the image of the $\lambda$-map.

Let $S_n^*$ denote the set of permutations $w$ in $S_n$ such that $w(1)=1.$ The set 
$S_n^*$ is a complete representatives of the coset space $C\backslash S_n.$ 
In particular, the cardinality of $\mathcal{B}_k^*$ is $(n-1)!$

%
%
Recall that there is a remarkable 
involution $\omega_k: \mathcal{B}_k\rightarrow \mathcal{B}_k, \; \mu\mapsto \mu^{\omega_k}$ due to Lapointe and Morse \cite{LapointeMorse2005JCTA}. One of the important properties is $\omega (s_\mu^{(k)})=s_{\mu^{\omega_k}}^{(k)}$ \cite[Theorem 38]{LapointeMorse2007} where $\omega$ is the involution on $\Lambda$ sending $s_\lambda$ to $s_{\lambda'}.$ 
One can see that $\omega_k$ preserves
$\mathcal{B}_k^*.$

\begin{example} The following tables give $\lambda(w)$ and its $k$-conjugate for $w\in S_n^*$ for $n=4,5.$
The asterisk sign indicates that the permutation is not Grassmannian.

$$
\begin{tiny}
\begin{array}{|l|l|l|}
\hline
w & \lambda(w)& \lambda(w)^{\omega_3}\\\hline
{1234} & \emptyset & \emptyset \\\hline
{1243} & (2) & (1,1)\\\hline
{1324} & (2,1) & (2,1)\\\hline
{1342} & (1) & (1)\\\hline
{1423} & (1,1)& (2)\\\hline
{1432}^* & (2,1,1)& (2,1,1)\\\hline
\end{array}
\qquad
\begin{array}{|l|l|l|}
\hline
w & \lambda(w)&\lambda(w)^{\omega_4}\\\hline
{12345} & \emptyset & \emptyset \\\hline
{12354} & (3) & (1,1,1)\\\hline
{12435} & (3,2) &(2,2,1)\\\hline
{12453} & (2)& (1,1)\\\hline
{12534} & (2,2)& (2,2)\\\hline
{12543}^* &{(3,2,2)}& (2,2,1,1,1)\\\hline
{13245} & (2,2,1) &(3,2)\\ \hline
13254^* & {(3,2,2,1)}&(3,2,1,1,1) \\\hline
{13425} & (3,1)& (2,1,1)\\\hline
{13452} & (1)& (1)\\ \hline
{13524} & (2,1) & (2,1)\\ \hline
13542^* & {(3,2,1)}&(2,2,1,1) \\\hline
{14235} & (2,1,1)& (3,1)\\ \hline
14253^* & {(3,2,1,1)}& (3,2,1,1)\\\hline
14325^* & {(3,2,2,1,1)}&(3,2,2,1,1) \\ \hline
14352^* & {(2,2,1,1)}& (3,2,1)\\ \hline
{14523} & (1,1)& (2)\\ \hline
14532^* & {(3,1,1)}& (2,1,1,1)\\ \hline
{15234} & (1,1,1)& (3) \\ \hline
15243^* & {(3,1,1,1)}& (3,1,1,1)\\ \hline
15324^* & {(3,2,1,1,1)}& (3,2,2,1)\\ \hline
15342^* & {(2,1,1,1)}& (3,1,1)\\\hline
15423^* & {(2,2,1,1,1)}& (3,2,2)\\\hline
15432^* & {(3,2,2,1,1,1)}&(3,2,2,1,1,1) \\\hline
\end{array}
\end{tiny}
$$
\end{example}

\subsection{Examples}

\begin{example} Let
$n=5$.
We first note some elements of $\tilde g_w$ are factored into a product of dual stable Grothendieck polynomials:
$$
\begin{tiny}
\begin{array}{|c|c|}
\hline
w & \tilde g_w\\\hline
12543&g_{1,1,1}\cdot g_{2,2}\\\hline
13254 & g_{1,1,1}\cdot g_{3,2}\\\hline
14532 &g_{1,1,1}\cdot g_2\\\hline
15243 &g_{1,1,1}\cdot g_3\\\hline 
15324 & g_{2,2,1}\cdot g_{3}\\\hline
15342 & g_{1,1}\cdot g_3\\\hline
15423 & g_{2,2}\cdot g_3\\\hline
15432 & g_{1,1,1}\cdot g_{2,2}\cdot g_3\\\hline
\end{array}
\end{tiny}
$$

Using Sage, we obtain the following expansion of $\tilde g_w$ in terms of $K$-theoretic $k$-Schur functions.
Here we only write the non-Grassmannian elements in $S_5^*.$ 

\begin{align*}
\tilde g_{12543}
&=g_{2,2,1,1,1}^{(4)}-g_{2,2,1,1}^{(4)},\\
\tilde g_{13254}&=
g_{3,2,1,1,1}^{(4)}-g_{3,2,1,1}^{(4)},\\
\tilde g_{13542}
&=g_{2,2,1,1}^{(4)}-g_{2,2,1}^{(4)},\\
\tilde g_{14253}
&=g_{3,2,1,1}^{(4)},\\
\tilde g_{14325}
&= g_{3,2,2,1,1}^{(4)}
-g_{3,2,1,1,1}^{(4)}
-g_{3,2,2,1}^{(4)}
+g_{3,2,1,1}^{(4)},\\
\tilde g_{14352}&=
g_{3,2,1}^{(4)}-g_{3,2}^{(4)},\\
\tilde g_{14532}&=
g_{2,1,1,1}^{(4)}-g_{2,1,1}^{(4)},
\\
\tilde g_{15243}&= g_{3,1,1,1}^{(4)},\\
\tilde g_{15324}&=g_{3,2,2,1}^{(4)}-g_{3,2,1,1}^{(4)}, \\
\tilde g_{15342}&= g_{3,1,1}^{(4)}-g_{3,1}^{(4)},\\
\tilde g_{15423}&=g_{3,2,2}^{(4)}-g_{3,2,1}^{(4)}, \\
\tilde g_{15432}
&=g_{3,2,2,1,1,1}^{(4)}-g_{3,2,1,1,1,1}^{(4)}
-2g_{3,2,2,1,1}^{(4)}
-g_{4,2,2,1}^{(4)}
+g_{3,2,1,1,1}^{(4)}
+g_{3,2,2,1}^{(4)}
-g_{3,2,1,1}^{(4)}.
\end{align*}
\end{example}

\bigskip

We finally provide one conjecture below. 

\begin{conjecture} Let $w_0\in S_n$ denote the longest element.
Then
$$
\tilde g_{w_0}=\prod_{i=1}^{n-2}g_{(n-1-i)^i}.
$$
\end{conjecture}

\bigskip

We are informed that 
a general formulation of $K$-theoretic Peterson isomorphism
valid for a simple simply connected $G$ is proposed by 
 Thomas Lam, Changzheng Li, Leonardo Mihalcea, and Mark Shimozono in \cite{LLMS}.
 We express thanks to them for 
communicating their work to us and for valuable discussions.
We are grateful to Michael Finkelberg for valuable discussions.  
He kindly showed us his interpretation of 
$\tau$-functions in terms of the Zastava space. 
During the preparation of the paper
we also benefitted from the 
helpful discussions with and comments
of many 
people, including 
Anders Buch, 
Rei Inoue, 
Hiroshi Iritani, 
Bumsig Kim, 
Cristian Lenart, 
Jennifer Morse,
Satoshi Naito, 
Hiraku Nakajima, 
Hiroshi Naruse, 
Kyo Nishiyama, 
Masatoshi Noumi, 
Kaisa Taipale, 
Kanehisa Takasaki, 
Motoki Takigiku, 
Vijay Ravikumar, and 
Chris Woodward.
We also thank Andrea Brini who drew our attention to \cite{KruglinskayaMarshakov2015}.
Special thanks are due to Tomoo Matsumura for showing us the proof of Proposition \ref{prop:LR}.   In order to discover and check Conjecture 2, we used the open source mathematical software Sage \cite{SageMath}.
The work was supported by JSPS KAKENHI [grant numbers 15K04832 to T.I., 26800062 to S.I., 16K05083 to T.M.].

\end{document}